\documentclass[12pt]{scrartcl}

\usepackage[]{amsmath, amssymb,amsfonts,amsthm,mathtools,braket,color,enumerate}

\usepackage{subcaption}
\usepackage{hhline}
\usepackage{url}
\usepackage{here}
\usepackage{tikz-cd}
\usepackage{tikz}
\usepackage{hyperref}
\hypersetup{
  colorlinks   = true, 
  urlcolor     = blue, 
  linkcolor    = blue, 
  citecolor   = red 
}
\usetikzlibrary{decorations}
\usetikzlibrary{arrows}
\tikzstyle{v} = [circle, draw, inner sep=2pt, minimum size=3pt, fill=black]
\tikzstyle{l} = [rectangle, draw, rounded corners]

\theoremstyle{plain}
\newtheorem{theorem}{Theorem}[section]
\newtheorem{lemma}[theorem]{Lemma}
\newtheorem{proposition}[theorem]{Proposition}

\theoremstyle{definition}
\newtheorem{definition}[theorem]{Definition}

\newtheorem{example}[theorem]{Example}

\newtheorem{remark}[theorem]{Remark}

\DeclareMathOperator{\Cox}{Cox}
\DeclareMathOperator{\Cat}{Cat}
\DeclareMathOperator{\Shi}{Shi}
\DeclareMathOperator{\Ish}{Ish}
\DeclareMathOperator{\Der}{Der}
\newcommand{\A}{\mathcal{A}}
\newcommand{\B}{\mathcal{B}}
\newcommand{\M}{\mathcal{M}}
\newcommand{\Z}{\mathbb{Z}}
\newcommand{\K}{\mathbb{K}}
\newcommand{\R}{\mathbb{R}}
\newcommand{\aff}{\stackrel{\mathrm{aff}}{=}}
 
\newcommand{\codim}{\operatorname{codim}}
\newcommand{\cc}{\mathbf{c}} 
\newcommand{\tbf}{\textbf}

\begin{document}

\title{Vertex-weighted Digraphs and Freeness of Arrangements Between Shi and Ish}
\author{
Takuro Abe
\thanks{
Institute of Mathematics for Industry, Kyushu University, Fukuoka 819-0395, Japan.
E-mail address: abe@imi.kyushu-u.ac.jp.
} \and
Tan Nhat Tran
\thanks{
Department of Mathematics, Hokkaido University, Sapporo, Hokkaido 060-0810, Japan. 
Current address: Fakult\"at f\"ur Mathematik, Ruhr-Universit\"at Bochum, D-44780 Bochum, Germany.
E-mail address: tan.tran@ruhr-uni-bochum.de. 
}\and
Shuhei Tsujie
\thanks{Department of Mathematics, Hokkaido University of Education, Asahikawa, Hokkaido 070-8621, Japan. 
E-mail address: tsujie.shuhei@a.hokkyodai.ac.jp.}
}

\date{\today}

\maketitle

\begin{abstract}
We introduce and study a digraph analogue of Stanley's $\psi$-graphical arrangements from the perspectives of combinatorics and freeness.
Our arrangements form a common generalization of various classes of arrangements in literature including the Catalan arrangement, the Shi arrangement, the Ish arrangement, and especially the arrangements interpolating between Shi and Ish recently introduced by Duarte and Guedes de Oliveira. 
The arrangements between Shi and Ish all are proved to have the same characteristic polynomial with all integer roots, thus raising the natural question of their freeness.
We define two operations on digraphs, which we shall call king and coking elimination operations and prove that subject to certain conditions on the weight $\psi$,  the operations preserve the characteristic polynomials and freeness of the associated  arrangements. 
As an application, we affirmatively prove that  the arrangements between Shi and Ish all are free, and among them only the Ish arrangement has supersolvable cone. 
\end{abstract}


    
\section{Introduction}
\quad
Let $ \mathbb{K} $ be a field and let $ \mathbb{K}^{\ell} $ be the  $ \ell $-dimensional vector space over $ \mathbb{K} $. 
An \textbf{arrangement} $ \mathcal{A} $ is a finite collection of hyperplanes in  $ \mathbb{K}^{\ell} $. 
We say that $ \mathcal{A} $ is \textbf{central} if every hyperplane in $ \mathcal{A} $ passes through the origin. 

Let  $ \mathcal{A} $ be an arrangement.
Define the \textbf{intersection poset} $ L(\mathcal{A}) $ of $ \mathcal{A} $ by 
\begin{align*}
L(\mathcal{A}) \coloneqq \Set{\bigcap_{H \in \mathcal{B}}H \neq \varnothing |  \mathcal{B} \subseteq \mathcal{A} },
\end{align*}
where the partial order is given by reverse inclusion $X\le Y\Leftrightarrow Y\subseteq X$ for $X, Y \in L(\A)$. 
We agree that $ \mathbb{K}^{\ell} $ is a unique minimal element in $ L(\mathcal{A}) $ as the intersection over the empty set. 
Thus $ L(\mathcal{A}) $ is a semi-lattice which can be equipped with the rank function $ r(X) \coloneqq \operatorname{codim}(X) $ for $X \in L(\mathcal{A})$. 
We also define the \textbf{rank} $r(\A)$ of $\A$ as the rank of a maximal element of $L(\A)$.
The intersection poset $ L(\mathcal{A}) $ is sometimes referred to as the  \textbf{combinatorics} of  $ \mathcal{A} $.

The \textbf{characteristic polynomial} $ \chi_{\mathcal{A}}(t) \in \mathbb{Z}[t] $ of $ \mathcal{A} $ is defined by
\begin{align*}
\chi_{\mathcal{A}}(t) \coloneqq \sum_{X \in L(\mathcal{A})}\mu(X)t^{\dim X}, 
\end{align*}
where $ \mu $ denotes the \textbf{M\"{o}bius function} $ \mu \colon L(\mathcal{A}) \to \mathbb{Z} $ defined recursively by 
\begin{align*}
\mu\left( \mathbb{K}^{\ell} \right) \coloneqq 1 
\quad \text{ and } \quad 
\mu(X) \coloneqq -\sum_{\substack{Y \in L(\mathcal{A}) \\ X \subsetneq Y}}\mu(Y). 
\end{align*}

Let $ \{x_{1}, \dots, x_{\ell}\} $ be a basis for the dual space $ \left(\mathbb{K}^{\ell}\right)^{\ast} $ and let $ S \coloneqq \mathbb{K}[x_{1}, \dots, x_{\ell}] $. 
The \textbf{defining polynomial} $Q(\A)$ of $\A$ is given by
$$Q(\A)\coloneqq \prod_{H \in \A} \alpha_H \in S,$$
where $ \alpha_H=a_1x_1+\cdots+a_\ell x_\ell+d$  $(a_i, d \in \mathbb{K})$ satisfies $H = \ker \alpha_H$. 
The operation of \textbf{coning} is a standard way to pass from any arrangement to a central one. 
The \textbf{cone} $ \textbf{c}\A$ over $\A$ is the central arrangement in $\mathbb{K}^{\ell+1}$ with the defining polynomial
$$Q(\textbf{c}\A)\coloneqq z\prod_{H \in \A} {}^h\alpha_H \in \mathbb{K}[x_1,\ldots, x_\ell,z],$$
where $ {}^h\alpha_H=a_1x_1+\cdots+a_\ell x_\ell+dz$ is the homogenization of $\alpha_H$, and $z=0$ is the \tbf{infinite hyperplane}, denoted $H_{\infty}$. 
The characteristic polynomials of $\A$ and $\textbf{c}\A$ are related by the following simple formula (e.g., \cite[Proposition 2.51]{OT92}):
$$\chi_{\textbf{c}\A}(t)=(t-1)\chi_\A(t).$$

The concept of free arrangements was defined by Terao for central arrangements \cite{T80}.
Given a central arrangement $ \mathcal{A} $, the \textbf{module $D(\mathcal{A}) $ of logarithmic derivations}  is defined by 
$$D(\A)\coloneqq  \{ \theta\in \Der(S) \mid \theta(\alpha_H) \in \alpha_H S \mbox{ for all } H \in \A\},$$
where $ \Der(S) $ denotes the set of derivations of $ S $ over $\mathbb{K}$.

\begin{definition}
A central arrangement $ \mathcal{A} $ is called \textbf{free} with the multiset $ \exp(\mathcal{A}) = \{d_{1}, \dots, d_{\ell}\} $ of \textbf{exponents} if $D(\A)$ is a free $S$-module with a homogeneous basis $ \{\theta_{1}, \dots, \theta_{\ell}\}$ such that $ \deg \theta_{i} = d_{i} $ for each $ i $.  
\end{definition}

Though the freeness was defined in an algebraic sense, it is related to the combinatorics of arrangements due to a remarkable result of Terao. 

\begin{theorem}[Factorization Theorem, e.g., \cite{T81}, {\cite[Theorem 4.137]{OT92}}]\label{thm:Factorization}
If $\A$ is free with $\exp(\A) =\{d_{1}, \dots, d_{\ell}\} $, then 
$$\chi_\A(t)= \prod_{i=1}^\ell (t-d_i).$$
\end{theorem}

In the remainder of the paper, unless otherwise explicitly stated, assume $\K=\R$.\footnote{For most applications we will only be concerned with arrangements that are originally defined over $\R$, though many of our arguments hold true for fields of characteristic $ 0 $.}
Our starting examples are two specific non-central arrangements in $\R^\ell$: the \textbf{Shi arrangement} $ \Shi(\ell) $ due to Shi \cite[Chapter 7]{Shi86},  and the \textbf{Ish arrangement} $ \Ish(\ell) $ due to Armstrong \cite{Arms13},
\begin{align*}
\Shi(\ell) &\coloneqq \Cox(\ell) \cup \Set{ x_{i}-x_{j}=1 | 1 \leq i < j \leq \ell}, \\
\Ish(\ell) &\coloneqq \Cox(\ell) \cup \Set{ x_{1}-x_{j} = i | 1 \leq i < j \leq \ell},
\end{align*}
where $\Cox(\ell) \coloneqq \Set{ x_{i}-x_{j}=0 | 1 \leq i < j \leq \ell} $ is the  \textbf{Coxeter arrangement of type $A$} (also known as the \textbf{braid arrangement}).

In general, the Shi arrangement and the Ish arrangement are combinatorially different in the sense that $ L(\Shi(\ell)) $ and $ L(\Ish(\ell)) $ are not isomorphic as posets. 
However,  these two have a number of interesting similarities. 
Some important properties are summarized below. 
\begin{theorem}[e.g., \cite{Headley97, A96, Arms13, AR12}]
The Shi arrangement and the Ish arrangement have the same characteristic polynomial:
\begin{align*}
\chi_{\Shi(\ell)}(t) = \chi_{\Ish(\ell)}(t) = t(t-\ell)^{\ell-1}. 
\end{align*}
\end{theorem}

If the integer $d$ appears $e\ge0$ times in a multiset $M$, we write $d^e \in M$.
\begin{theorem}[\cite{Atha98, AST17,Yo04}]
\label{thm:SI-free}
The cones $ \mathbf{c}\Shi(\ell) $ and $ \mathbf{c}\Ish(\ell) $ are free with exponents $\{0^1, 1^1, \ell^{\ell-1}\}$. 
However,  $ \mathbf{c}\Ish(\ell) $ is supersolvable while $ \mathbf{c}\Shi(\ell) $ is not supersolvable if $\ell \ge 3$.
\end{theorem}

In this paper we will focus on a larger family containing $\Shi(\ell) $ and  $\Ish(\ell) $, the  \textbf{arrangements ``between Shi and Ish"} introduced recently by Duarte and Guedes de Oliveira in their study of Pak-Stanley labeling of the regions of real arrangements \cite{DG18, DG19}. 
Let $2 \leq k \leq \ell$. Define 
$$
\mathcal{A}_{\ell}^{k} \coloneqq \Cox(\ell) 
\cup \Set{ x_{1}-x_{j}=i  | 1 \leq i < j \leq \ell, i < k}
\cup \Set{ x_{i}-x_{j}=1  | k \leq i < j \leq \ell}. 
$$
We call $ \mathcal{A}_{\ell}^{k}$ in this paper the   \textbf{$(k,\ell)$-Shi-Ish arrangement}. 
These arrangements interpolate between $\Shi(\ell) $ and  $\Ish(\ell) $ as $k$ varies, and we can view $\Shi(\ell)= \mathcal{A}_{\ell}^1 = \mathcal{A}_{\ell}^{2} $ and $ \Ish(\ell)= \mathcal{A}_{\ell}^{\ell}  $ as the extreme arrangements in this family. 
The following is a notable feature of the $(k,\ell)$-Shi-Ish arrangements.
\begin{theorem}[\cite{DG18}]\label{Duarte--Guedes-de-Oliveira}
If $2 \leq k \leq \ell$, then the characteristic polynomial of $ \mathcal{A}_{\ell}^{k} $ is 
\begin{align*}
\chi_{\mathcal{A}_{\ell}^{k}}(t) = t(t-\ell)^{\ell-1}. 
\end{align*}
\end{theorem}

Thus Theorem \ref{Duarte--Guedes-de-Oliveira} together with Terao Factorization Theorem (Theorem \ref{thm:Factorization}) naturally raises the question of freeness for $\cc \mathcal{A}_{\ell}^{k} $. 
One of the new results derived from the present paper is an affirmative answer to this question with more information on their supersolvability.
\begin{theorem}\label{main1}
If $2 \leq k \leq \ell$, then the cone $\textbf{c} \mathcal{A}_{\ell}^{k} $ is free with exponents $\{0^1, 1^1, \ell^{\ell-1}\}$. 
Moreover, if $\ell \ge 3$ then $\textbf{c} \mathcal{A}_{\ell}^{k} $ is not supersolvable except when $k=\ell$.
\end{theorem}

To unveil the essence of the freeness and (non-)supersolvability of $\textbf{c} \mathcal{A}_{\ell}^{k} $, we will study the problem in greater generality. 
We define a class of arrangements associated to a vertex-weighted digraph (or directed graph), containing $ \Shi(\ell) $, $ \Ish(\ell) $, and their interpolations. 
Furthermore, we introduce two operations on the vertex-weighted digraphs under which the characteristic polynomial and freeness are preserved. 
We then obtain Theorem \ref{main1} by applying a sequence of these operations to $ \Shi(\ell) $. 
More interestingly, all of the $(k,\ell)$-Shi-Ish arrangements $ \mathcal{A}_{\ell}^{k} $ appear as the members in the operation sequence, thus giving a new insight into how these arrangements arise naturally as intermediate arrangements between Shi and Ish.

The remainder of the paper is organized as follows. 
In \S \ref{sec:digraph}, we define the arrangements associated to vertex-weighted digraphs, called the $ \psi $-digraphical arrangements, and the operations producing new arrangements from a given one as mentioned above. 
We aim at preservation of both characteristic polynomial and freeness of arrangements and in the end of this section we specify three conditions on the weight $\psi$ that are suitable for our purposes. 
In \S \ref{sec:stability-chi}, we prove the first main result in the paper that under the first two indicated conditions, the operations preserve the  characteristic polynomials of the $ \psi $-digraphical arrangements. 
In particular, it gives a new proof of Theorem \ref{Duarte--Guedes-de-Oliveira}. 
In \S \ref{sec:freeness-stable}, we prove our second main result that the operations preserve the  freeness of the $ \psi $-digraphical arrangements if all three conditions are satisfied. 
The proof of Theorem \ref{main1} will be presented in \S \ref{subsec:freeness-stable} as an application of the second main result.

\section{$ \psi $-digraphical arrangements}
\label{sec:digraph}
\quad
For integers $a\le b$ and $\ell\ge1$, denote $[a,b]\coloneqq \{n\in\Z\mid a \le n\le b\}$ and $[\ell]\coloneqq [1,\ell]$. 
Let $ G = (V_{G}, E_{G}) $ be a digraph on $ V_{G} = [\ell]$. 
A directed edge $(i,j) \in E_{G}$ is considered to be \tbf{directed from $i$ to $j$} ($i \longrightarrow j$). 
Let $ \psi \colon [\ell] \to 2^{\mathbb{Z}} $ be a map, called \tbf{(integral) weight} map, defined by $ \psi(i) = [a_{i}, b_{i}]\subseteq 2^{\mathbb{Z}}$, where $a_{i}\le b_{i}$ are integers for every $ i \in [\ell] $. 
We call the pair $(G,\psi)$ a \textbf{vertex-weighted digraph}.

\begin{definition}
Let $(G,\psi)$ be a vertex-weighted digraph.
Define the \textbf{$ \psi $-digraphical arrangement} $ \mathcal{A}(G,\psi) $ in $ \mathbb{R}^{\ell} $ by 
$$
\mathcal{A}(G,\psi) \coloneqq 
\Cox(\ell) 
\cup \Set{x_{i}-x_{j}=1  | (i,j) \in E_{G}}
\cup \Set{ x_{i}=c  | c \in \psi(i), 1 \leq i   \leq \ell}. 
$$
\end{definition}

We sometimes use the notation $ \mathcal{A}(G,\psi(i)) $ for $ \mathcal{A}(G,\psi) $ when we want to emphasize the precise evaluation  of $\psi$ at $ i \in [\ell] $. 
In particular, if $\psi$ is a constant map with output $U$, i.e., $ \psi(i) = U\subseteq 2^{\mathbb{Z}}  $ for every $ i \in [\ell] $, we write $ \mathcal{A}(G,U)$. 

Two (central) hyperplane arrangements $\A$ and $\mathcal{B}$ in $\K^\ell$ are said to be \textbf{(linearly) affinely equivalent} if there is an invertible (linear) affine endomorphism $\varphi: \K^\ell \to \K^\ell$ such that $\B=\varphi(\A)=\{\varphi(H)\mid H\in \A\}$. 
In particular, the intersection posets of two affinely equivalent arrangements are isomorphic. 
Also, one can prove that the freeness is preserved under linear equivalence. 
In the rest of the paper, we will often ``identify" affinely equivalent arrangements and when necessary use the notation $\A \aff \mathcal{B}$ for two  affinely equivalent arrangements $\A$ and $\mathcal{B}$. 
Note that for such non-central $\A$ and $\mathcal{B}$, the cones $\cc\A$ and $\cc\mathcal{B}$ are linearly equivalent.

\begin{remark}
Given a simple undirected graph $ G = ( [\ell], E_{G}) $ and  a weight map $ \psi \colon [\ell] \to 2^{\mathbb{Z}} $, the \textbf{$ \psi $-graphical arrangement} $\mathcal{A}_{G,\psi}  $  in $ \mathbb{R}^{\ell} $ was defined by Stanley \cite{St15} as follows:
$$
\mathcal{A}_{G,\psi} \coloneqq 
  \Set{x_{i}-x_{j}=0 | \{i,j\} \in E_{G}}
\cup \Set{x_{i}=c  | c \in \psi(i), 1 \leq i   \leq \ell}. 
$$

Our arrangement $ \mathcal{A}(G,\psi) $ can be regarded as a digraph analogue of Stanley $ \psi $-graphical arrangement, and up to affine equivalence neither of these two concepts is a specialization of the other. 
Indeed, a $ \psi $-digraphical arrangement necessarily contains the Coxeter  hyperplanes $\Cox(\ell)$, while a $ \psi $-graphical arrangement does not (e.g., take the empty arrangement).
Freeness and supersolvability (see \S \ref{sec:stability-chi} for definition) of the cones over $ \psi $-graphical arrangements are proved to be equivalent and can be characterized completely in terms of weighted elimination ordering \cite{MS15, ST19}. 
The Shi arrangement is a $ \psi $-digraphical arrangement (see Example \ref{ex:empty} below) whose associated cone is free but not supersolvable \cite{Atha98} (see also Theorems \ref{thm:SI-free} and  \ref{main1}).

We find it highly nontrivial to characterize the freeness or supersolvability of the $ \psi $-digraphical arrangements in full generality but we will show in our second main result (Theorem \ref{thm:freeness-stable}) that there is a certain subclass general enough for our purposes whose freeness can be described. 
\end{remark}

Now we define some simple but important classes of digraphs for our discussion later.
\begin{definition}
The \textbf{transitive tournament} on $ [\ell]$, denoted $ T_{\ell} $, is given by 
\begin{align*}
E_{T_{\ell}} = \Set{(i,j) | 1 \leq i < j \leq \ell}. 
\end{align*}
The \textbf{complete digraph} on $ [\ell]$, denoted $ K^{\ast}_{\ell} $, is given by 
\begin{align*}
E_{K^{\ast}_{\ell}} = \Set{(i,j) | i,j \in [\ell], \ i \neq j}. 
\end{align*}
The \textbf{edgeless digraph} on $ [\ell]$, denoted  $ \overline{K^{\ast}_{\ell}} $, is given by 
$$ E_{\overline{K^{\ast}_{\ell}}} = \varnothing.$$
\end{definition}
See Figure \ref{fig:tce} for depiction of the digraphs above when $\ell=4$. 

\begin{figure}[htbp]
\centering
\begin{subfigure}{.35\textwidth}
  \centering
\begin{tikzpicture}[scale=1]
\draw (0,3) node[v](1){} node[above]{$ \textbf{1}$};
\draw (-1,2) node[v](2){} node[left]{$ \textbf{2}$};
\draw (0,1) node[v](3){} node[below]{$ \textbf{3}$};
\draw (1,2) node[v](4){} node[right]{$ \textbf{4}$};
\draw[>=Stealth,->] (1)--(2);
\draw[>=Stealth,->] (1)--(3);
\draw[>=Stealth,->] (1)--(4);
\draw[>=Stealth,->] (2)--(3);
\draw[>=Stealth,->] (2)--(4);
\draw[>=Stealth,->] (3)--(4);
\end{tikzpicture}
  \caption*{$T_4$}
  \label{fig:tt}
\end{subfigure}%
\begin{subfigure}{.35\textwidth}
  \centering
\begin{tikzpicture}[scale=1]
\draw (0,3) node[v](1){} node[above]{$ \textbf{1}$};
\draw (-1,2) node[v](2){} node[left]{$ \textbf{2}$};
\draw (0,1) node[v](3){} node[below]{$ \textbf{3}$};
\draw (1,2) node[v](4){} node[right]{$ \textbf{4}$};
\draw[>=Stealth,->, bend right = 10] (1) to (2);
\draw[>=Stealth,->, bend right = 10] (1) to (3);
\draw[>=Stealth,->, bend right = 10] (1) to (4);
\draw[>=Stealth,->, bend right = 10] (2) to (1);
\draw[>=Stealth,->, bend right = 10] (2) to (3);
\draw[>=Stealth,->, bend right = 10] (2) to (4);
\draw[>=Stealth,->, bend right = 10] (3) to (1);
\draw[>=Stealth,->, bend right = 10] (3) to (2);
\draw[>=Stealth,->, bend right = 10] (3) to (4);
\draw[>=Stealth,->, bend right = 10] (4) to (1);
\draw[>=Stealth,->, bend right = 10] (4) to (2);
\draw[>=Stealth,->, bend right = 10] (4) to (3);
\end{tikzpicture}
  \caption*{$K^{\ast}_{4} $}
  \label{fig:cd}
\end{subfigure}%
\begin{subfigure}{.35\textwidth}
  \centering
\begin{tikzpicture}[scale=1]
\draw (0,3) node[v](1){} node[above]{$ \textbf{1}$};
\draw (-1,2) node[v](2){} node[left]{$ \textbf{2}$};
\draw (0,1) node[v](3){} node[below]{$ \textbf{3}$};
\draw (1,2) node[v](4){} node[right]{$ \textbf{4}$};
\end{tikzpicture}
  \caption*{$  \overline{K^{\ast}_{4}} $}
  \label{fig:ed}
\end{subfigure}
\caption{From left to right: the  transitive tournament, the complete digraph, and the edgeless digraph  on $4$ vertices.}
\label{fig:tce}
\end{figure}
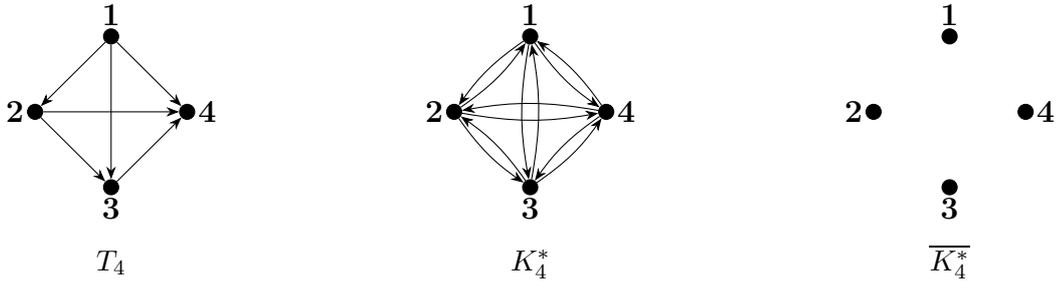

The following are some simple specializations of  the arrangement $ \mathcal{A}(G,\psi) $.

\begin{example}
\label{ex:empty}
Clearly, $ \mathcal{A}(\overline{K^{\ast}_{\ell}}, \varnothing) =\Cox(\ell)$ and $ \mathcal{A}(T_{\ell}, \varnothing) = \Shi(\ell) $. Recall that the \textbf{Catalan arrangement} is defined by 
$$\Cat(\ell) \coloneqq \Cox(\ell) \cup \Set{x_{i}-x_{j}=1, -1  | 1 \leq i < j \leq \ell}.$$
Thus, $ \mathcal{A}(K^{\ast}_{\ell},\varnothing) =\Cat(\ell)$.
\end{example}

\begin{remark}
\label{rem:cat}
$\cc\Cat(\ell)$ is known to be free with exponents $\{0,1, \ell+1, \ell+2,\ldots,2\ell-1\}$, e.g., \cite{ER96, Yo04}. 
As $G$ varies between $\overline{K^{\ast}_{\ell}}$ and $K^{\ast}_{\ell}$, the arrangements $ \mathcal{A}(G, \varnothing) $ can be considered as  intermediate arrangements interpolating between $\Cox(\ell)$ and  $\Cat(\ell)$. 
A characterization for freeness of $ \mathcal{A}(G, \varnothing) $ was conjectured by Athanasiadis \cite[Conjecture 6.6]{Ath00} and it was completely settled in the works of Nuida, Numata, and the first author \cite[Theorem 5.3]{ANN09}, \cite[Corollary 1.1]{Abe12}. 
\end{remark}

 Let $\A_1$ and $\A_2$ be arrangements in $\K^m$ and $\K^n$, respectively. 
The  \textbf{product} arrangement $\A_1\times \A_2$ is an arrangement in  $\K^m\oplus\K^n \simeq\K^{m+n}$ defined by
\begin{align*}
\A_1\times \A_2 \coloneqq \{ H_1 \oplus \K^n \mid H_1 \in \A_1\} \cup  \{ \K^m \oplus H_2 \mid H_2 \in \A_2\}.
\end{align*}

Let $ \Phi_{\ell} $ denote the $ \ell $-dimensional \textbf{empty arrangement}, that is, the arrangement in $ \mathbb{K}^{\ell} $ consisting of no hyperplanes. 

Now we give some less trivial specializations (up to affine equivalence) of our arrangement $ \mathcal{A}(G,\psi) $. 
Let $ G = ( [\ell], E_{G}) $ be a digraph and let ${\psi} \colon [\ell] \to 2^{\mathbb{Z}}, i \mapsto [a_i,b_i] $ be a weight map on $[\ell]$.
Define an arrangement $\widetilde{\A}(G, \psi)$ in $ \mathbb{R}^{\ell+1} $ with an additional coordinate $ x_{0} $ by
\begin{align*}
\widetilde{\A}(G,  \psi) \coloneqq  \Cox(\ell) &\cup \Set{x_{i}-x_{j}=1 | (i,j) \in E_{G}, 1 \leq i < j \leq \ell } \\
& \cup \Set{  x_{0} -x_{i}=c  | c \in  \psi(i), 1 \leq i   \leq \ell }. 
\end{align*}
Thus  via the coordinate change $ x_{0} \mapsto x_{0} $, $ x_{i}-x_{0} \mapsto x_{i} $ $ (1 \leq i \leq \ell )$,
\begin{align*}
\widetilde{\A}(G,  \psi) \aff \mathcal{A}(G ,-\psi)\times \Phi_1, \text{ where } (-\psi)(i)\coloneqq [-b_{i}, -a_{i}] \text{ for } i \in [\ell]. 
\end{align*}

\begin{example}
The Ish arrangement $ \Ish(\ell +1) $ consists of the following hyperplanes: 
\begin{align*}
x_{i} - x_{j} &= 0 \quad (1 \leq i < j \leq \ell + 1), \\
x_{1} - x_{j} &= i \quad (1 \leq i < j \leq \ell + 1). 
\end{align*}
Changing coordinates $ \{x_{1}, \dots, x_{\ell + 1}\} $ to $ \{x_{0}, \dots, x_{\ell} \} $ with $ x_{i} \mapsto x_{i-1} $, we obtain 
\begin{align*}
x_{i} - x_{j} &= 0 \quad (1 \leq i < j \leq \ell), \\
x_{0} - x_{i} &= c \quad (c \in [0,i]). 
\end{align*}
Therefore $ \Ish(\ell + 1) = \widetilde{\mathcal{A}}(\overline{K^{\ast}_{\ell}}, [0,i]) \aff \mathcal{A}(\overline{K^{\ast}_{\ell}}, [-i,0]) \times \Phi_{1} $. 
\end{example}

\begin{example}
\label{ex:aff}
For arbitrary weight map ${\psi} \colon [\ell] \to 2^{\mathbb{Z}} $, the arrangement $\widetilde{\A}(\overline{K^{\ast}_{\ell+1}},  \psi)$ is essentially the \textbf{$N$-Ish arrangement} \cite[Definition 1.2]{AST17} (when defined over $\R$) due to Suyama, the first author and the third author. 
Thus, every $ N $-Ish arrangement can be represented as a $ \psi $-digraphical arrangement. 
Note that every $N$-Ish arrangement is also affinely equivalent to a Stanley $\psi$-graphical arrangement.
\end{example}

\begin{example}
$ \Cox(\ell + 1) = \widetilde{\A}(\overline{K^{\ast}_{\ell}},\{0\}) \aff \mathcal{A}(\overline{K^{\ast}_{\ell}}, \{0\}) \times \Phi_{1} $. 
\end{example}

\begin{example}
$ \Shi(\ell + 1) = \widetilde{\A}(T_{\ell},[0,1]) \aff \mathcal{A}(T_{\ell}, [-1,0]) \times \Phi_{1} $. 
\end{example}

\begin{example}
$ \Cat(\ell + 1) = \widetilde{\A}(K^{\ast}_{\ell},[-1,1]) \aff \mathcal{A}(K^{\ast}_{\ell}, [-1,1]) \times \Phi_{1} $. 
\end{example}

In a similar way, the $ (k, \ell) $-Shi-Ish arrangement $ \mathcal{A}_{\ell}^{k} $ can be represented as a $ \psi $-digraphical arrangement. 
However, for convenience of later discussion, we give the following expression of $ \mathcal{A}_{\ell}^{k} $. 

\begin{proposition}
\label{prop:Akl modified}
Let $ 1 \leq k \leq \ell $ and 
let $ T_{\ell}^{k} $ be a digraph on $ [\ell] $ with edge set
\begin{align*}
E_{T_{\ell}^{k}} \coloneqq \Set{(i,j) | 1 \leq i < j \leq \ell - k + 1}. 
\end{align*}
Define $ \psi_{\ell}^{k} \colon [\ell] \to 2^{\mathbb{Z}} $ by $ \psi_{\ell}^{k}(i) \coloneqq [-\min\{\ell-i+1, k\}, 0] $. 
Then 
\begin{align*}
\mathcal{A}_{\ell+1}^{k+1} \aff \mathcal{A}(T_{\ell}^{k}, \psi_{\ell}^{k}) \times \Phi_{1}. 
\end{align*}
\end{proposition}
\begin{proof}
The arrangement $ \mathcal{A}_{\ell+1}^{k+1} $ consists of the following hyperplanes:
\begin{align*}
x_{i}-x_{j} &= 0 \qquad (1 \leq i < j \leq \ell+1), \\
x_{i}-x_{j} &= 1 \qquad (k+1 \leq i < j \leq \ell+1), \\
x_{1}-x_{j} &= i \qquad (1 \leq i < j \leq \ell+1, i<k+1). 
\end{align*}
Applying the coordinate change $x_{i}\mapsto x_{i-1}$ ($1 \leq i   \leq \ell+1)$  yields 
\begin{align*}
x_{i}-x_{j} &= 0 \quad (1 \leq i < j \leq \ell), \\
x_{i}-x_{j} &= 1 \quad (k \leq i < j \leq \ell), \\
x_{0}-x_{j} &= c \quad (1 \leq   j \leq \ell, c \in [0,\min\{j,k\}] ). 
\end{align*}
Moreover, changing the coordinates by $ x_{0} \mapsto x_{0} $ and $ x_{i}-x_{0} \to x_{i} $ for $ 1 \leq i \leq \ell $, we obtain 
\begin{align*}
x_{i}-x_{j} &= 0 \quad (1 \leq i < j \leq \ell), \\
x_{i}-x_{j} &= 1 \quad (k \leq i < j \leq \ell), \\
x_{i} &= c \quad (1 \leq   i \leq \ell, c \in [-\min\{i,k\}, 0] ). 
\end{align*}
Finally applying $ x_{i} \mapsto x_{\sigma(i)} $, where $ \sigma $ is the permutation on $ [\ell] $ defined by 
\begin{align*}
\sigma \coloneqq \begin{pmatrix}
1 & 2 & \dots & k-1 & k & k+1 & \dots & \ell \\
\ell & \ell -1 & \dots & \ell -k & 1 & 2 & \dots & \ell-k+1
\end{pmatrix}, 
\end{align*}
we have 
\begin{align*}
x_{i}-x_{j} &= 0 \quad (1 \leq i < j \leq \ell), \\
x_{i}-x_{j} &= 1 \quad (1 \leq i < j \leq \ell-k+1), \\
x_{i} &= c \quad (1 \leq   i \leq \ell, c \in [-\min\{\ell-k+1,k\}, 0] ), 
\end{align*}
the desired result. 
\end{proof}

\begin{definition}
A vertex $ v $ in a digraph $ G $ is called a \textbf{king} (resp., \textbf{coking}) if $ (v,u) \in E_{G} $ (resp., $ (u,v) \in E_{G} $) for every $ u \in V_{G}\setminus\{v\} $. 
\end{definition}
\begin{remark}
A $ d $-king in a digraph $G$ is sometimes known as a vertex that can reach every other
vertex in $G$ by a directed path of length at most $d$.
Though a king often refers to a $ 2 $-king, in this paper, a king means a $ 1 $-king.  
\end{remark}

The following definition of digraph operations is crucial.
\begin{definition}\label{deformation coking}
Let $(G,\psi)$ be a vertex-weighted digraph, that is, $ G $ is a digraph on $ [\ell] $ and $ \psi \colon [\ell] \to 2^{\mathbb{Z}} $ is a  map such that $ \psi(i) = [a_{i}, b_{i}]\subseteq 2^{\mathbb{Z}}$ where $a_{i}\le b_{i}$ are integers for every $ i \in [\ell] $. Let $v$ be a vertex in $G$.
\begin{enumerate}[(1)]
\item Suppose that $v$ is a coking in $ G $. 
The \textbf{coking elimination operation (CEO)} on $G$ (w.r.t. $v$) is a construction of a new vertex-weighted digraph $ (G^{\prime}, \psi^{\prime}) $ where $ G^{\prime}  =( [\ell] , E_{G'})$ is a digraph   and $ \psi^{\prime} $ is a weight map given by
\begin{align*}
E_{G^{\prime}} \coloneqq E_{G}\setminus\Set{(i,v) | i \in [\ell]\setminus\{v\}}, \qquad
\psi^{\prime}(i) \coloneqq \begin{cases}
[a_{i}-1, b_{i}] & (i \in [\ell]\setminus\{v\}), \\
[a_{v}, b_{v}] & (i=v). 
\end{cases}
\end{align*}
\item Dually, suppose that  $v$ is a king in $ G $. 
The \textbf{king elimination operation (KEO)} on $G$ (w.r.t. $v$)  produces a new vertex-weighted digraph $ (G'', \psi'') $ given by
\begin{align*}
E_{G''} \coloneqq E_{G}\setminus\Set{(v,i) | i \in [\ell]\setminus\{v\}}, \qquad
\psi''(i) \coloneqq \begin{cases}
[a_{i}, b_{i}+1] & (i \in [\ell]\setminus\{v\}), \\
[a_{v}, b_{v}] & (i=v). 
\end{cases}
\end{align*}
\end{enumerate}
\end{definition}

Hereafter when we say that we apply CEO or KEO on a $ \psi $-digraphical arrangement $ \mathcal{A}(G,\psi) $, we actually mean we apply that operation on the underlying vertex-weighted digraph $(G,\psi)$.

\begin{remark}
\label{rem:interchange}
For a digraph $ G $ on $ [\ell] $, denote by $ G^{\mathrm{conv}} $  the \textbf{converse} of $G$. 
Namely, $ G^{\mathrm{conv}} $ is the digraph on $ [\ell] $ obtained by reversing the direction of each edge of $ G $. 
Thus $ \mathcal{A}(G,\psi) \aff  \mathcal{A}(G^{\mathrm{conv}}, -\psi) $ via $ x_{i} \mapsto -x_i $.
Taking the converse of a digraph interchanges kings and cokings. 
Throughout the paper, our arguments will be stated mostly in terms of cokings, however, one can easily derive the analogous results in terms of kings by the equivalence $ \mathcal{A}(G,\psi) \aff  \mathcal{A}(G^{\mathrm{conv}}, -\psi) $. 
\end{remark}

We close this section by defining some specific conditions on the weight $\psi$ that will be essential for our main results in the next sections.
\begin{definition}\label{def:conditions}
Let $(G,\psi)$ be a vertex-weighted digraph on $ [\ell] $ and let $v$ be a vertex in $G$.
Define the following conditions on $\psi$ (w.r.t. $v$).
\begin{enumerate}
\item[(\tbf{C})]  \label{cond:C} $ 0 \in \psi(i) $ for every $ i \in [\ell] $ and $ [a_{i}, b_{i}-1] \subseteq \psi(v) $ for each $ i \in [\ell]\setminus\{v\} $. 
\item[(\tbf{K})]  $ 0 \in \psi(i) $ for every $ i \in [\ell] $ and $ [a_{i}+1, b_{i}] \subseteq \psi(v) $ for each $ i \in [\ell]\setminus\{v\}$. 
\item[(\tbf{Z})] There exists $ n_{0}\in\Z_{\ge0} $ such that $n_0+2\le |\psi(v)| \le n_0+3$ and $n_0+1\le |\psi(i)| \le n_0+3$ for every $ i \in [\ell]\setminus\{v\}$. 
\end{enumerate}
\end{definition}

Roughly speaking, in view of Remark \ref{rem:interchange}, Conditions \ref{def:conditions}(\tbf{C}) and  \ref{def:conditions}(\tbf{K}) are  dual to each other, while Condition \ref{def:conditions}(\tbf{Z}) is self-dual. 

\section{Stability of characteristic polynomials}
\label{sec:stability-chi}
In this section we prove our first main result that under suitable conditions, the CEO and KEO preserve the  characteristic polynomials of the $ \psi $-digraphical arrangements. 

\begin{theorem}[Stability of characteristic polynomials]
\label{coincidence characteristic polynomials}
Let $(G,\psi)$ be a vertex-weighted digraph on $ [\ell] $ and let $v$ be a vertex in $G$.
\begin{enumerate}[(1)]
\item Assume that $v $ is a coking in $ G $. 
Let $ (G^{\prime}, \psi^{\prime}) $ be the graph obtained from $(G,\psi)$ by applying the CEO w.r.t. the coking $v$ (Definition \ref{deformation coking}). 
If Condition \ref{def:conditions}(\tbf{C}) is satisfied, then $$ \chi_{\mathcal{A}(G,\psi)}(t) = \chi_{\mathcal{A}(G^{\prime},\psi^{\prime})}(t).$$

\item Assume that $ v $ is a king in $ G $. 
Let $ (G'', \psi'') $ be the graph obtained from $(G,\psi)$ by applying the KEO w.r.t. the king $v$ (Definition \ref{deformation coking}). 
If Condition \ref{def:conditions}(\tbf{K})  is satisfied, then $$ \chi_{\mathcal{A}(G,\psi)}(t) = \chi_{\mathcal{A}((G'', \psi'')}(t).$$
\end{enumerate}
 \end{theorem}

The key ingredient is the finite field method.
A hyperplane arrangement is said to be \textbf{integral} if the equations defining the hyperplanes of the arrangement have integer coefficients. 
Let $\A$ be an integral hyperplane arrangement in $\R^\ell$ and let $p$ be a prime number. 
The arrangement $\A$ gives rise to an arrangement $ \mathcal{A}_{p} =\A \otimes \mathbb{F}_{p} $ in $ \mathbb{F}_{p}^{\ell} $ defined by regarding the defining equation of each hyperplane in $\A$ as lying over $\mathbb{F}_{p} $. 
Note that $L(\A_p) \simeq L(\A)$ if $p$ is sufficiently large.

\begin{theorem}[Finite field method, e.g., \cite{CR70,A96,BE97,BS98,KTT08}]
\label{Crapo--Rota finite field method}
Let $\A$ be an integral hyperplane arrangement in $\R^\ell$ and let $p$ be a large enough prime number. Then
$$ \chi_{\mathcal{A}}(p)=
\left| \,\mathbb{F}_{p}^{\ell} \setminus \bigcup_{H \in \mathcal{A}_{p}}H\,\right|.
$$
\end{theorem}

We are ready to prove Theorem \ref{coincidence characteristic polynomials}. 

\begin{proof}[\tbf{Proof of Theorem \ref{coincidence characteristic polynomials}}]
In view of Remark \ref{rem:interchange}, it suffices to prove (1).
We use the finite field method. 
Let $ p $ be a prime large enough and let $ \mathcal{A}_{p} $ and $ \mathcal{A}^{\prime}_{p} $ denote the corresponding  arrangements over $ \mathbb{F}_{p} $ of $ \mathcal{A}(G,\psi) $ and $ \mathcal{A}(G^{\prime},\psi^{\prime}) $. 
Let $ M $ and $ M^{\prime} $ denote the complements of these arrangements, that is, 
\begin{align*}
M \coloneqq \mathbb{F}_{p}^{\ell} \setminus \bigcup_{H \in \mathcal{A}_{p}}H
\quad \text{ and } \quad
M^{\prime} \coloneqq \mathbb{F}_{p}^{\ell} \setminus \bigcup_{H \in \mathcal{A}^{\prime}_{p}}H. 
\end{align*}
In order to  prove $ \chi_{\mathcal{A}(G,\psi)}(t) = \chi_{\mathcal{A}(G^{\prime},\psi^{\prime})}(t)$, we will show $ |M| = |M^{\prime}| $. 

 Let $ \boldsymbol{x} = (x_{1}, \dots, x_{\ell})\in \mathbb{F}_{p}^{\ell} $. 
Note that $ x_{i} \neq x_{j} $ for distinct $ i, j $ if $ \boldsymbol{x} \in M $ or $ \boldsymbol{x} \in M^{\prime} $ since both $ \mathcal{A}_{p} $ and $ \mathcal{A}^{\prime}_{p} $ contain the hyperplane $ x_{i}-x_{j} = 0 $. 
The intersection $ M \cap M^{\prime} $ is given by 
\begin{align*}
M \cap M^{\prime} &= \Set{\boldsymbol{x} \in \mathbb{F}_{p}^{\ell} | \boldsymbol{x} \not\in H \text{ for any } H \in \mathcal{A}_{p} \cup \mathcal{A}^{\prime}_{p}} \\
&= \Set{\boldsymbol{x} \in M | x_{i} \neq a_{i}-1 \text{ for any } i \in [\ell]\setminus\{v\} } \\
&= \Set{\boldsymbol{x} \in M^{\prime} | x_{i} - x_{v} \neq 1 \text{ for any } i \in [\ell]\setminus\{v\} }. 
\end{align*}
Therefore the set differences $ M\setminus M^{\prime} $ and $ M^{\prime}\setminus M $ are  given by 
\begin{align*}
M\setminus M^{\prime} &= \Set{\boldsymbol{x} \in M | x_{i} = a_{i} - 1 \text{ for some } i \in [\ell]\setminus\{v\}}, \\
M^{\prime}\setminus M &= \Set{\boldsymbol{y} \in M^{\prime} | y_{i} - y_{v} = 1 \text{ for some } i \in [\ell]\setminus\{v\}}. 
\end{align*}

Now, we will construct a bijection between $M\setminus M^{\prime}$ and $M^{\prime}\setminus M $. 
Define the total order $ \preceq $ on $ \mathbb{F}_{p} $ induced by the usual order on the representatives $ \{0, 1, \dots, p-1 \} $ of $ \mathbb{F}_{p} $. 

Let $ \boldsymbol{x} \in M\setminus M^{\prime} $. 
Since $0\in \psi(i) = [a_{i}, b_{i}] $ and $x_i \notin  [a_{i}, b_{i}]$ for each $ i \in [\ell] $, we have $ 0 \preceq b_{i} \prec x_{i} \preceq a_{i}-1 \preceq p-1 $. 

There exists $n \in [\ell]\setminus\{v\} $ such that $ x_{n} = a_{n}-1 $. 
Suppose $ x_{n} \preceq x_{v} $. 
Thus $ a_{n}-1 = x_{n} \preceq x_{v} \preceq a_{v}-1 \preceq a_{n}-1 $, where the last relation follows from $ [a_{n}, b_{n}-1] \subseteq \psi(v) $. 
Hence $ x_{v} = x_{n} $, which is a contradiction. 
Therefore $ x_{n} \succ x_{v} $ and the set $ \Set{x_{k} | x_{k} \succ x_{v}} $ is non-empty. 

Let $ x_{k_{0}} :=\min \Set{x_{k} | x_{k} \succ x_{v}}$.
Set $ d \coloneqq x_{k_{0}}-x_{v} \succ 0 $. 
Since $v$ is a coking in $ G $, we have $ x_{k_{0}} - x_{v} \neq 1 $. 
Hence $ d \succeq 2 $. 
Define $ f \colon M\setminus M^{\prime} \to M^{\prime}\setminus M $ by $ f(\boldsymbol{x}) =\boldsymbol{y} $ where 
\begin{align*}
y_{i} \coloneqq \begin{cases}
x_{i} & (x_{i} \preceq x_{v}),  \\
x_{i}-d+1 & (x_{i} \succ x_{v}). 
\end{cases}
\end{align*}
Note that $ y_{v} = x_{v} $ while $ y_{i} \succ y_{v} $ if and only if $ x_{i} \succ x_{v} $. 

We will show $ \boldsymbol{y} \in M^{\prime}\setminus M $. 
First suppose $ y_{i} = y_{j} $. 
If $ y_{i} = y_{j} \preceq y_{v} $, then $ x_{i} = y_{i} = y_{j} = x_{j} $, hence $ i = j $. 
If $ y_{i} = y_{j} \succ y_{v} $, then $ x_{i} = y_{i}+d-1 = y_{j}+d-1 = x_{j} $, hence $ i=j $. 
Therefore $ y_{i} \neq y_{j} $ for any distinct $ i, j \in [\ell]$. 

Second, we will show $ b_{i} \prec y_{i} \prec a_{i}-1 $ for each $ i \in [\ell]\setminus\{v\} $. 
Note that $ b_{v} \prec y_{v} = x_{v} \preceq a_{v}-1 $. 
If $ y_{i} \prec y_{v} $, then $ b_{i} \prec x_{i} = y_{i} \prec y_{v} \preceq a_{v}-1 \preceq a_{i}-1 $. 
If $ y_{i} \succ y_{v} $, then $ b_{i} \preceq b_{v}+1 \preceq y_{v} \prec y_{i} = x_{i}-d+1 \preceq x_{i}-1 \prec  a_{i}-1 $. 
Thus $ b_{i} \prec y_{i} \prec a_{i}-1 $ for each $ i \in [\ell]\setminus\{v\} $. 

Next suppose $ (i,j) \in E_{G^{\prime}} $. 
Note that $ j \neq v $ by definition. 
When $ y_{i} \prec y_{j} $, the condition $ y_{i}-y_{j}=1 $ implies $ y_{i}=0 $ and $ y_{j} = p-1 $, which is a contradiction since $ y_{i} \succ b_{i} \succeq 0 $. 
Consider  $ y_{i} \succ y_{j} $. 
If $ y_{i} \succ y_{j} \succ y_{v} $ or $ y_{v} \succeq y_{i} \succ y_{j} $, then $ y_{i}-y_{j} = x_{i}-x_{j} \neq 1 $. 
If $ y_{i} \succ y_{v} \succ y_{j} $, then $ y_{i}-y_{j} \succeq 2 $, hence $ y_{i}-y_{j} \neq 1 $. 

By the discussion above, $ \boldsymbol{y} \in M^{\prime} $. 
Furthermore, since $ y_{k_{0}}-y_{v}=( x_{k_{0}}-d+1)-x_{v}  = 1 $, we have $ \boldsymbol{y} \in M^{\prime}\setminus M $. 

To show that $ f $ is bijective, we will construct the inverse of $ f $. 
Let $ \boldsymbol{y} \in M^{\prime}\setminus M $. 
There exists $ n \in [\ell]\setminus\{v\} $ such that $ y_n-y_{v}=1 $. 
Since $ y_n \neq 0 $, we have $ y_n \succ y_{v} $. 
Hence the set $ \Set{a_k-y_k | y_k \succ y_{v}} $ is non-empty. 
Set $ c \coloneqq \min\Set{a_k-y_k | y_k \succ y_{v}} $. 
Define $ g \colon M^{\prime}\setminus M  \to  M\setminus M^{\prime}$ by $ g(\boldsymbol{y}) =\boldsymbol{x} $ where 
\begin{align*}
x_{i} \coloneqq \begin{cases}
y_{i} & (y_{i} \preceq y_{v}), \\
y_{i}+c-1 & (y_{i} \succ y_{v}). 
\end{cases}
\end{align*}
In a similar way, one can show $ \boldsymbol{x} \in M\setminus M^{\prime} $. 
Moreover, it is easily seen that $f\circ g = g\circ f = \mathrm{id}$. 
Therefore $ f $ is a bijection, hence $ |M\setminus M^{\prime}| = |M^{\prime}\setminus M| $. 
Thus $ |M| = |M^{\prime}| $. 
\end{proof}

Many applications of Theorem \ref{coincidence characteristic polynomials} are related to the concept of supersolvable arrangements  due to Stanley \cite{St72}, which we will recall shortly.
Let $ \mathbb{K} $ be any field. 
Let $\A$ be a central arrangement in $\K^\ell$. 
For each $X \in L(\A)$, we define the \textbf{localization} of $\A$ on $X$  by 
$${\A}_X \coloneqq \{ K \in {\A} \mid X \subseteq K\} \subseteq \A,$$
and define the \textbf{restriction} ${\A}^{X}$ of ${\A}$ to $X$ by 
$${\A}^{X} \coloneqq \{ K \cap X \mid K \in{\A }\setminus {\A}_X\}.$$

An element  $ X \in L(\mathcal{A}) $ is said to be \textbf{modular} if $ X+Y \in L(\mathcal{A}) $ for all $ Y \in L(\mathcal{A}) $. 
A modular element of corank $ 1 $ is called a \textbf{modular coatom}. 
If $ X \in L(\mathcal{A}) $ is a modular coatom, we call the localization $ \mathcal{A}_{X}$ a modular coatom of $\A$ as well. 

 \begin{definition}
 \label{def:supersolvable}
A central arrangement $\A$ of rank $r$ is called \textbf{supersolvable} if there exists a chain of arrangements, called an \textrm{M}-chain,
$$\varnothing  =  \mathcal{A}_{X_0} \subseteq \mathcal{A}_{X_1} \subseteq \dots \subseteq \mathcal{A}_{X_{r}} = \A,$$
in which $ \mathcal{A}_{X_{i}} $ is a modular coatom of $ \mathcal{A}_{X_{i+1}} $ for each $0 \le i \le r-1$. 
\end{definition}

The following result is useful to check whether  an element is a modular coatom. 
\begin{proposition}[{\cite[Theorem 4.3]{BEZ90}}]
 \label{prop:ss}
Let $ \mathcal{A} $ be a central arrangement and let $ X \in L(\mathcal{A}) $. 
Then $ \mathcal{A}_{X} $ is a modular coatom if and only if for any distinct $ H, H^{\prime} \in \mathcal{A}\setminus \mathcal{A}_{X} $, there exists $ H^{\prime\prime} \in \mathcal{A}_{X} $ such that $ H\cap H^{\prime} \subseteq H^{\prime\prime} $. 
\end{proposition}

 We now recall some basic facts and classical results on supersolvable and free arrangements. 
 
 \begin{theorem}[{\cite[Theorem (4.2)]{JT84}} or {\cite[Theorem 4.58]{OT92}}]
 \label{thm:exp-ss}
If $\A$ is supersolvable, then $\A$ is free. 
Furthermore, if $\A$ has an \textrm{M}-chain $\varnothing  =  \mathcal{A}_{X_0} \subseteq \mathcal{A}_{X_1} \subseteq \dots \subseteq \mathcal{A}_{X_{r(\A)}} = \A,$ then
$\exp(\A) = \{0^{\ell-r(\A)}, d_1, \ldots, d_{r(\A)}\}$ where $d_i = | \mathcal{A}_{X_{i}} \setminus  \mathcal{A}_{X_{i-1}}|$.
 \end{theorem}

  \begin{theorem}[{\cite[Proposition 3.2]{St72}} and  {\cite[Theorem 4.37]{OT92}}]
 \label{thm:localization-ss-free}
 If $\A$ is  supersolvable (resp., free), then $\A_X$ is  supersolvable (resp., free) for  any $X \in L(\A)$.
 \end{theorem}

 \begin{proposition}[{\cite[Proposition 2.14 and Lemma 2.50]{OT92}}]
 \label{prop:prod-L(A) isom}
 Let $\A_1$ and $\A_2$ be arrangements. Define a partial order on the
set $L(\A_1) \times L(\A_2)$ of pairs $(X_1, X_2)$ with $X_i \in L(\A_i)$ by 
$$(X_1, X_2) \le (Y_1, Y_2)\Leftrightarrow  X_1 \le Y_1 \mbox{ and } X_2 \le Y_2.$$
There exists a natural isomorphism of lattices 
$$\pi: L(\A_1) \times L(\A_2) \to L(\A_1\times \A_2)$$ 
given by $\pi (X_1, X_2) = X_1\oplus X_2.$ 
In particular, 
$$\chi_{\A_1\times \A_2}(t) =\chi_{\mathcal{A}_1}(t)\cdot\chi_{\mathcal{A}_2}(t) .$$
\end{proposition}
 
  \begin{proposition}[{\cite[Proposition 4.28]{OT92}} and {\cite[Proposition 2.5]{HR14}}]
 \label{prop:prod-L(A) free}
 Let $\A_1$ and $\A_2$ be arrangements. The product arrangement $\A_1\times \A_2$ is supersolvable (resp., free) if and only if both $\A_1$ and $\A_2$ are supersolvable (resp., free). In this case, $ \exp (\A) =  \exp (\A_1) \cup  \exp (\A_2) $.
 \end{proposition}

An arrangement $ \mathcal{A} $ in $ \mathbb{K}^{\ell} $ is called \textbf{essential} if $ r(\mathcal{A}) = \ell $. 
Any arrangement $ \mathcal{A} $ of rank $r$ in $\K^\ell$ can be written as the product of an essential arrangement $ \mathcal{A}^{\mathrm{ess}} $ and the $(\ell-r)$-dimensional empty arrangement $ \Phi_{\ell-r}$. 
We call $ \mathcal{A}^{\mathrm{ess}} $ the \textbf{essentialization} of $ \mathcal{A} $ (see e.g., \cite[\S1.1]{St07}). 
Note that $ \mathcal{A} $ is free if and only if $ \mathcal{A}^{\mathrm{ess}} $ is free and also note that $ \chi_{\mathcal{A}}(t) = t^{\ell-r}\chi_{\mathcal{A^{\mathrm{ess}}}}(t) $. 

 \begin{proposition} \label{prop:modular coatom}
 Let $\A$ be a central arrangement and let  $X\in L(\mathcal{A})$ be a modular coatom. 
The following statements hold.
\begin{enumerate}[(1)]
\item $\chi_{\mathcal{A}^{\mathrm{ess}}}(t) = \left(t-|\mathcal{A}\setminus\mathcal{A}_{X}|\right)\cdot\chi_{\mathcal{A}_{X}^{\mathrm{ess}}}(t) $. 
\item $ \mathcal{A} $ is free if and only if $ \mathcal{A}_{X} $ is free. Furthermore, $ \exp (\A^{\mathrm{ess}}) = \exp (\A_X^{\mathrm{ess}} ) \cup \{|\mathcal{A}\setminus\mathcal{A}_{X}|\} $. 
\end{enumerate}
\end{proposition}
\begin{proof}
(1) follows from {\cite[Theorem 2]{St71}}. 
We only need to show that if $ \mathcal{A}_{X} $ is free, then $ \mathcal{A} $ is free. 
Note that by \cite[Lemma 2.62]{OT92} there exists $H\in \A$ such that $\A^H$ is supersolvable and $L(\A^H) \simeq L( \mathcal{A}_{X_{r(\A)-1}})$. 
Apply \cite[Theorem 1.1]{Abe16}. 
\end{proof}

A simplicial vertex in a simple undirected graph is a vertex whose neighbors form a clique (every two neighbors are adjacent). 
The following is a version of simplicial vertices in vertex-weighted digraphs.
\begin{definition}
Let $(G,\psi)$ be a vertex-weighted digraph on $ [\ell] $. 
Let $v$ be a vertex in $G$ and let $X_v\in L(\mathbf{c}\mathcal{A}(G,\psi) )$ be the intersection of the following hyperplanes:
\begin{align*}
z&= 0 \\
x_{i}-x_{j} &= 0 \qquad (i,j\in [\ell]\setminus\{v\}), \\
x_{i}-x_{j} &= z \qquad ((i,j) \in E_{G}, i,j\in [\ell]\setminus\{v\}), \\
 x_{i} &= cz\qquad (c \in \psi(i), i\in [\ell]\setminus\{v\}). 
\end{align*}
The vertex $ v $ is said to be \textbf{simplicial} in $ (G,\psi) $ if $X_v$ is a modular coatom of $ \mathbf{c}\mathcal{A}(G,\psi) $.
\end{definition}

Let  $ G\setminus v $ denote the subgraph obtained from $G$ by removing $ v $ and the edges incident to and from $v$.  
Thus $( \mathbf{c}\mathcal{A}(G,\psi) )_{X_v}=\mathbf{c}\mathcal{A}\left(G\setminus v, \psi|_{[\ell]\setminus\{v\}}\right) \times \Phi_{1}$. 

\begin{proposition}
\label{prop:simplicial}
Let $ v $ be a simplicial vertex of a vertex-weighted digraph $ (G,\psi) $ on $ [\ell] $. 
Then the following statements hold. 
\begin{enumerate}[(1)]
\item $ \chi_{\mathcal{A}(G,\psi)}(t) = \left(t - (|\psi(v)|+e+\ell-1)\right)\chi_{\mathcal{A}(G\setminus v, \psi|_{[\ell]\setminus\{v\}})}(t) $, where $ e $ denotes the number of edges incident to and from $ v $. \\
\item $ \mathbf{c}\mathcal{A}(G,\psi) $ is free if and only if $ \mathbf{c}\mathcal{A}(G\setminus v,\psi|_{[\ell]\setminus v}) $. 
\end{enumerate}
\end{proposition}
\begin{proof}
It follows immediately from Proposition \ref{prop:modular coatom}. 
\end{proof}

\begin{proposition}
 \label{prop:isolated}
 Let $(G,\psi)$ be a vertex-weighted digraph on $ [\ell] $. 
Let $ v $ be an isolated vertex (a vertex that is not incident to or from any edge) of $ G $. 
If the weight of $v$ is minimal in $(G,\psi)$, i.e., $ \psi(v) \subseteq \psi(i) $ for all $ i \in [\ell] $, then $ v $ is simplicial in $ (G,\psi) $. 
\end{proposition}
\begin{proof}
The statement is an easy application of Proposition \ref{prop:ss}. 
\end{proof}

Now we give an important application of Theorem \ref{coincidence characteristic polynomials} that will be used later for the proof of Theorem \ref{main1}.

\begin{theorem}
 \label{thm:Akl}
The sequence $ \Shi(\ell)=\mathcal{A}_{\ell}^{2}, \mathcal{A}_{\ell}^{3}, \dots, \mathcal{A}_{\ell}^{\ell} = \Ish(\ell) $ of arrangements can be derived by CEOs. 
Moreover, we obtain an alternative proof of Theorem \ref{Duarte--Guedes-de-Oliveira}.
 \end{theorem}
\begin{proof}
First recall from Proposition \ref{prop:Akl modified} that $ \mathcal{A}_{\ell+1}^{k+1} \aff \mathcal{A}(T_{\ell}^{k}, \psi_{\ell}^{k}) \times \Phi_{1} $.
We will show that $ (T_{\ell}^{k+1}, \psi_{\ell}^{k+1}) $ is obtained from $ (T_{\ell}^{k}, \psi_{\ell}^{k}) $ by using CEO. 
Note that each vertex $ i \in [\ell-k+2, \ell] $ is isolated in both $ T_{\ell}^{k} $ and $ T_{\ell}^{k+1} $. 

The vertex $ v \coloneqq \ell-k+1 $ is a coking of $ T_{\ell}^{k}[\ell-k+1] $, the induced subgraph of $ T_{\ell}^{k} $ over $ [\ell-k+1] $. 
Now apply the CEO to $ T_{\ell}^{k}[\ell-k+1] $ with respect to the coking $ v $, 
we obtain $ (T_{\ell}^{k}[\ell-k+1])^{\prime} = T_{\ell}^{k+1}[\ell-k+1] $ and 
\begin{align*}
(\psi_{\ell}^{k})^{\prime}(v) &= [-\min\{\ell-v+1, k\}, 0] = [-k,0] \\
&= [-\min\{\ell-v+1, k+1\}, 0] = \psi_{\ell}^{k+1}(v). 
\end{align*}
Moreover, for every $ i \in [\ell-k] $
\begin{align*}
(\psi_{\ell}^{k})^{\prime}(i) &= [-\min\{\ell-i+1, k\}-1, 0] = [-k-1,0] \\
&= [-\min\{\ell-i+1, k+1\}, 0] = \psi_{\ell}^{k+1}(i). 
\end{align*}
Therefore $ ((T_{\ell}^{k}[\ell-k+1])^{\prime}, (\psi_{\ell}^{k}|_{[\ell-k+1]})^{\prime}) = (T_{\ell}^{k+1}[\ell-k+1], \psi_{\ell}^{k+1}|_{[\ell-k+1]}) $. 
By adding isolated vertices $ \ell-k+2, \dots, \ell $ to each of $ T_{\ell}^{k}[\ell-k+1] $ and $ T_{\ell}^{k+1}[\ell-k+1] $, we obtain $ T_{\ell}^{k+1} $ from $ T_{\ell}^{k} $ by using CEO. 

Next we will show that $ \chi_{\mathcal{A}(T_{\ell}^{k},\psi_{\ell}^{k})}(t) = \chi_{\mathcal{A}(T_{\ell}^{k+1},\psi_{\ell}^{k+1})}(t) $ for every $ k \in [\ell-1] $. 
For every $ i \in [\ell-k+1] $, $ \psi_{\ell}^{k}(i) = [-k,0] = \psi_{\ell}^{k}(v) $, where $ v = \ell-k+1 $ is the coking of $ T_{\ell}^{k}[\ell-k+1] $. 
Then Condition \ref{def:conditions}(\tbf{C}) is satisfied and hence $ \chi_{\mathcal{A}(T_{\ell}^{k}[\ell-k+1],\psi_{\ell}^{k}|_{[\ell-k+1]})}(t) = \chi_{\mathcal{A}(T_{\ell}^{k+1}[\ell-k+1],\psi_{\ell}^{k+1}|_{[\ell-k+1]})}(t) $ by Theorem \ref{coincidence characteristic polynomials}. 

For any vertex $ i \in [\ell-k+2, \ell] $, it is easy to see that $ \psi_{\ell}^{k}(i) = \psi_{\ell}^{k+1}(i) $ and 
\begin{align*}
\psi_{\ell}^{k}(1) &\supseteq \psi_{\ell}^{k}(2) \supseteq \dots \supseteq \psi_{\ell}^{k}(\ell), \\
\psi_{\ell}^{k+1}(1) &\supseteq \psi_{\ell}^{k+1}(2) \supseteq \dots \supseteq \psi_{\ell}^{k+1}(\ell). 
\end{align*}
Therefore the vertex $ i $ is simplicial in each of $ T_{\ell}^{k}[\{1,2, \dots, i\}] $ and $ T_{\ell}^{k+1}[\{1,2, \dots, i\}] $ by Proposition \ref{prop:isolated}. 
Using Proposition \ref{prop:simplicial}, we have $ \chi_{\mathcal{A}(T_{\ell}^{k},\psi_{\ell}^{k})}(t) = \chi_{\mathcal{A}(T_{\ell}^{k+1},\psi_{\ell}^{k+1})}(t) $. 

In particular, 
\begin{align*}
\chi_{\mathcal{A}(T_{\ell}^{\ell},\psi_{\ell}^{\ell})} = \prod_{i=1}^{\ell}\left(t-(|\psi_{\ell}^{\ell}(i)|+i-1)\right) = \prod_{i=1}^{\ell}\left(t-((\ell+2-i)+i-1)\right) = (t-(\ell+1))^{\ell}. 
\end{align*}
Thus, 
\begin{align*}
\chi_{\mathcal{A}_{\ell}^{2}}(t) = \chi_{\mathcal{A}_{\ell}^{3}}(t) = \dots = \chi_{\mathcal{A}_{\ell}^{\ell}}(t) = t \, \chi_{\mathcal{A}(T_{\ell-1}^{\ell-1}, \psi_{\ell-1}^{\ell-1})}(t) = t(t-\ell)^{\ell-1}. 
\end{align*}
\end{proof}

Figure \ref{fig:(k,5)} depicts the CEO sequence in Theorem \ref{thm:Akl} that applies to $  \Shi(5)$ and produces all $(k,5)$-Shi-Ish arrangements.
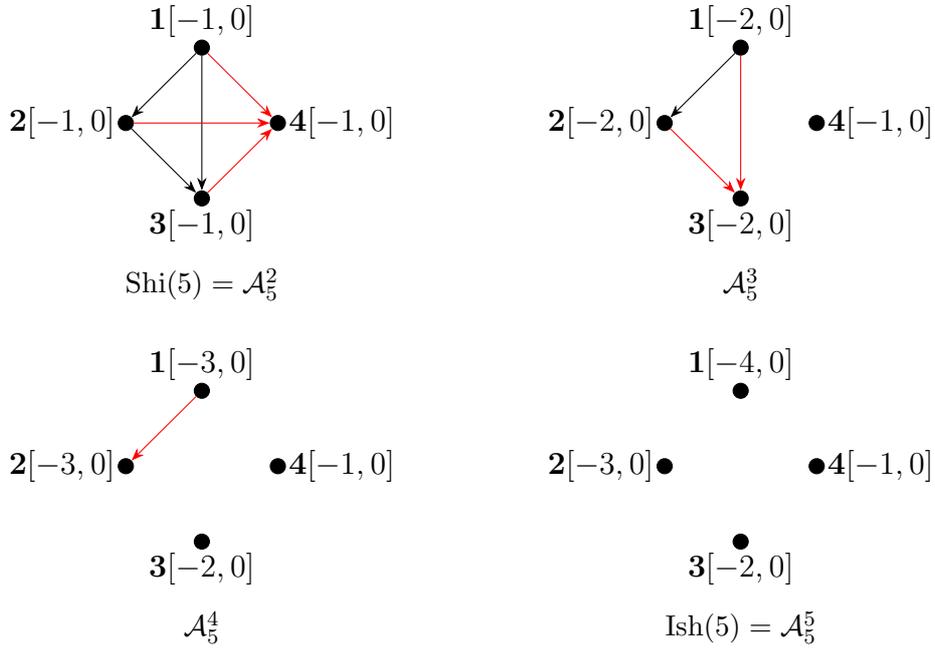
\begin{figure}[htbp]
\centering
\begin{subfigure}{.45\textwidth}
  \centering
\begin{tikzpicture}[scale=1]
\draw (0,3) node[v](1){} node[above]{$ \textbf{1}[-1,0] $};
\draw (-1,2) node[v](2){} node[left]{$ \textbf{2}[-1,0] $};
\draw (0,1) node[v](3){} node[below]{$ \textbf{3}[-1,0] $};
\draw (1,2) node[v](4){} node[right]{$ \textbf{4}[-1,0] $};
\draw[>=Stealth,->] (1)--(2);
\draw[>=Stealth,->] (1)--(3);
\draw[>=Stealth,->,red] (1)--(4);
\draw[>=Stealth,->] (2)--(3);
\draw[>=Stealth,->,red] (2)--(4);
\draw[>=Stealth,->,red] (3)--(4);
\end{tikzpicture}
  \caption*{$ \Shi(5) = \mathcal{A}_{5}^{2} $}
\end{subfigure}%
\begin{subfigure}{.45\textwidth}
  \centering
\begin{tikzpicture}[scale=1]
\draw (0,3) node[v](1){} node[above]{$\textbf{1} [-2,0] $};
\draw (-1,2) node[v](2){} node[left]{$ \textbf{2}[-2,0] $};
\draw (0,1) node[v](3){} node[below]{$ \textbf{3} [-2,0] $};
\draw (1,2) node[v](4){} node[right]{$ \textbf{4} [-1,0] $};
 \draw[>=Stealth,->] (1)--(2);
\draw[>=Stealth,->,red] (1)--(3);
\draw[>=Stealth,->,red] (2)--(3);
\end{tikzpicture}
  \caption*{$ \mathcal{A}_{5}^{3} $}
\end{subfigure}%

\bigskip
\begin{subfigure}{.45\textwidth}
  \centering
\begin{tikzpicture}[scale=1]
\draw (0,3) node[v](1){} node[above]{$\textbf{1} [-3,0] $};
\draw (-1,2) node[v](2){} node[left]{$\textbf{2} [-3,0] $};
\draw (0,1) node[v](3){} node[below]{$ \textbf{3} [-2,0] $};
\draw (1,2) node[v](4){} node[right]{$ \textbf{4} [-1,0] $};
\draw[>=Stealth,->,red] (1)--(2);
\end{tikzpicture}
  \caption*{$ \mathcal{A}_{5}^{4} $}
\end{subfigure}%
\begin{subfigure}{.45\textwidth}
  \centering
\begin{tikzpicture}[scale=1]
\draw (0,3) node[v](1){} node[above]{$\textbf{1} [-4,0] $};
\draw (-1,2) node[v](2){} node[left]{$\textbf{2} [-3,0] $};
\draw (0,1) node[v](3){} node[below]{$ \textbf{3} [-2,0] $};
\draw (1,2) node[v](4){} node[right]{$ \textbf{4} [-1,0] $};
\end{tikzpicture}
  \caption*{$ \Ish(5) = \mathcal{A}_{5}^{5} $}
\end{subfigure}
\caption{$(k,5)$-Shi-Ish arrangements $ \mathcal{A}_{5}^{k}$ $(2 \le k \le 5)$. Here a vertex is in bold symbol placed next to its weight.}
\label{fig:(k,5)}
\end{figure}

\section{Stability  of freeness}
\label{sec:freeness-stable}
In this section we prove our second main result that under suitable conditions, the CEO and KEO preserve the  freeness of the $ \psi $-digraphical arrangements. 

\begin{theorem}[Stability of freeness]
\label{thm:freeness-stable}
Let $(G,\psi)$ be a vertex-weighted digraph on $ [\ell] $ and let $v$ be a vertex in $G$.
\begin{enumerate}[(1)]
\item Assume that $v $ is a coking in $ G $. 
Let $ (G^{\prime}, \psi^{\prime}) $ be the graph obtained from $(G,\psi)$ by applying the CEO w.r.t. the coking $v$. 
Suppose that Conditions \ref{def:conditions}(\tbf{C}) and  \ref{def:conditions}(\tbf{Z}) are satisfied. 
Then $\cc \mathcal{A}(G,\psi) $ is free if and only if $\cc \mathcal{A}(G^{\prime},\psi^{\prime}) $ is free. 
\item Assume that $ v $ is a king in $ G $. 
Let $ (G'', \psi'') $ be the graph obtained from $(G,\psi)$ by applying the KEO w.r.t. the king $v$. 
Suppose that Conditions \ref{def:conditions}(\tbf{K}) and  \ref{def:conditions}(\tbf{Z}) are satisfied. 
Then $\cc \mathcal{A}(G,\psi) $ is free if and only if $\cc \mathcal{A}(G'',\psi'') $ is free. 
\end{enumerate}
 \end{theorem}

The proof of Theorem \ref{thm:freeness-stable} requires a careful analysis of the freeness of codimension-$3$ localizations and of the Ziegler restrictions of $\cc \mathcal{A}(G,\psi) $. Let us briefly recall these concepts and hint the key ingredient in the proof. 

A \tbf{multiarrangement} is a pair $(\A, m)$ where $\A$ is an arrangement  and $m$ is a map $m : \A \to \Z_{\ge0}$, called \tbf{multiplicity} map. 
Let $\A$ be a central arrangement  in $\K^\ell$ and let $(\A, m)$ be a multiarrangement.
The defining polynomial $Q(\A,m)$ of $(\A, m)$ is given by 
$$Q(\A,m)\coloneqq \prod_{H \in \A} \alpha^{m(H)}_H \in S= \mathbb{K}[x_{1}, \dots, x_{\ell}].$$
When $m(H) = 1$ for every $H \in \A$, $(\A,m)$ is simply a hyperplane arrangement. 
The \tbf{module $D(\A,m)$ of logarithmic derivations} of  $(\A, m)$ is defined by
$$D(\A,m)\coloneqq  \{ \theta\in \Der(S) \mid \theta(\alpha_H) \in \alpha^{m(H)}_HS \mbox{ for all } H \in \A\}.$$

We say that $(\A, m)$  is \tbf{free}  with the multiset $ \exp(\A, m) = \{d_{1}, \dots, d_{\ell}\} $ of \textbf{exponents}  if $D(\A,m)$ is a free $S$-module with a homogeneous basis $ \{\theta_{1}, \dots, \theta_{\ell}\}$  such that $ \deg \theta_{i} = d_{i} $ for each $ i $. 
It is known that  $(\A, m)$ is always free for $\ell\le 2$ \cite[Corollary 7]{Z89}.

Let $H \in \A$. The \tbf{Ziegler restriction} $(\A^H,m^H)$ of $\A$ onto $H$ is a multiarrangement defined by $m^H(X):=|\A_X|-1$ for $X \in \A^H$. 
We say that $\A$ is \tbf{locally free in codimension three along $H$} if $\A_X$ is free for all $X\in L(\A)$
with $X\subseteq H$ and $\codim(X)=3$. 

The key ingredient is a characterization for freeness, one direction is a classical result of Ziegler, the reverse direction is due to Yoshinaga and the first author.
\begin{theorem}[{\cite[Theorem 2.2]{Yo04}}, {\cite[Theorem 4.1]{AY13}}, {\cite[Theorem 11]{Z89}}]
\label{thm:Yoshinaga's criterion}
Let $\A$ be a central arrangement in $\K^\ell$ with $\ell \ge 3$ and let $H\in \A$. Then
$\A$ is free  with $\exp(\A) = \{1, d_2,\ldots, d_\ell\}$ if and only if the Ziegler restriction  $(\A^H,m^H)$ is free  with $\exp (\A^H, m^H)=\{d_2,\ldots, d_\ell\}$ and $\A$ is  locally free in codimension three along $H$.
\end{theorem}

Subsections  \ref{subsec:Ziegler Restrictions} and  \ref{subsec:Local Freeness}  below are devoted to showing that subject to the indicated conditions the freeness of Ziegler restrictions  and the local freeness are stable under CEO and KEO. 
The proof of Theorem \ref{thm:freeness-stable} will be presented in Subsection \ref{subsec:freeness-stable}.

\subsection{Freeness of Ziegler restrictions}
\label{subsec:Ziegler Restrictions}

Our goal in this subsection is to prove the following lemma.

\begin{lemma}[Stability of freeness of Ziegler restrictions]
\label{coincidence ziegler restrictions}
Let $(G,\psi)$ be a vertex-weighted digraph on $ [\ell] $ and let $v$ be a vertex in $G$. 
Let $ \mathcal{M} $ denote the Ziegler restriction onto the infinite hyperplane of the cone  $\cc \mathcal{A}(G,\psi) $.
Suppose that Condition \ref{def:conditions}(\tbf{Z}) is satisfied. 

\begin{enumerate}[(1)]
\item Assume that $v $ is a coking in $ G $. 
Let $ (G^{\prime}, \psi^{\prime}) $ be the graph obtained from $(G,\psi)$ by applying the CEO w.r.t. the coking $v$. 
Let $ \mathcal{M}^{\prime} $ denote the Ziegler restriction  onto the infinite hyperplane of the cone $\cc \mathcal{A}(G^{\prime}, \psi^{\prime}) $.
Then $ \mathcal{M} $ is free if and only if $ \mathcal{M}^{\prime} $ is free. 
\item Assume that $ v $ is a king in $ G $. 
Let $ (G'', \psi'') $ be the graph obtained from $(G,\psi)$ by applying the KEO w.r.t. the king $v$. 
Let $ \mathcal{M}''$ denote the Ziegler restriction  onto the infinite hyperplane of the cone $\cc \mathcal{A}(G'', \psi'') $.
Then $ \mathcal{M} $ is free if and only if $ \mathcal{M}'' $ is free. 
\end{enumerate}
\end{lemma}

As we will see, the Ziegler restrictions we are concerned with are related to special multiarrangements arising from signed graphs. 
A \textbf{signed graph} (e.g., \cite{Ha53}) is a triple $ \Gamma = (V_{\Gamma}, E_{\Gamma}^{+}, E_{\Gamma}^{-}) $, where 
$ V_{\Gamma} $ is a finite set, and $ E_{\Gamma}^{+} $ and $ E_{\Gamma}^{-} $ are disjoint subsets of the set of all unordered pairs of $ V_{\Gamma} $. 
The  \textbf{signature} $\epsilon= \epsilon_{\Gamma} $ of a signed graph $\Gamma$ is the map $ \epsilon \colon \binom{V_{\Gamma}}{2} \to [-1, 1] $ defined by 
$$
\epsilon(i,j) \coloneqq \begin{cases}
1 & \mbox{if } \{v_i,v_j\} \in E_{\Gamma}^{+}, \\
-1 & \mbox{if } \{v_i,v_j\}\in E_{\Gamma}^{-}, \\
0 & \text{otherwise},
\end{cases}
$$
where $ \{v_i,v_j\}$ denotes the undirected edge between $v_i$ and $v_j$.
Conversely, any map $ \epsilon \colon \binom{V}{2} \to [-1, 1] $ where $V$ is a finite set defines a signed graph $ \Gamma(\epsilon) $ whose signature is  $ \epsilon $ itself. 

A signed graph $ \Gamma = ([\ell], E_{\Gamma}^{+}, E_{\Gamma}^{-}) $ is \tbf{signed-eliminable} \cite[Definition 0.2]{ANN09} if for every three vertices $i,j,k$ with $i<k,j<k$, the induced subgraph $G[\{i,j,k\}]$ satisfies the following conditions.
\begin{enumerate}[(i)]
 \item  For $\sigma \in \{+,-\}$, if $\{i,k\} \in E_{\Gamma}^\sigma$ and $\{j,k\} \in E_{\Gamma}^\sigma$, then $\{i,j\} \in E_{\Gamma}^\sigma$.
  \item  For $\sigma \in \{+,-\}$, if $\{k,i\} \in E_{\Gamma}^\sigma$ and $\{i,j\} \in E_{\Gamma}^{-\sigma}$, then $\{k,j\} \in E_{\Gamma}^{-\sigma}$.
 \end{enumerate}
 
\begin{definition}
Let $ \Gamma $ be a signed graph on $V_{\Gamma}= [\ell] $. 
Let $ k $ be a positive integer and let $ \boldsymbol{n}=(n_{1}, \dots, n_{\ell}) $ be an $\ell$-tuple of non-negative integers. 
Let $\K$ be a field of characteristic zero.
Define a multiarrangement $(\A,m)$\footnote{This multiarrangement is a special case of the \tbf{multi-braid arrangement} defined in  \cite{ANN09}.}
 in $\K^\ell$ with $\A=\Cox(\ell)$ and $m=m_{k, \boldsymbol{n},\Gamma}$ where $ m(x_{i}-x_{j}=0) \coloneqq 2k+n_{i}+n_{j}+\epsilon_{\Gamma} (i,j) $ for  $ 1 \leq i < j \leq \ell $. 
\end{definition}

\begin{theorem}[{\cite[Theorem 0.3]{ANN09}}]
\label{thm:ANN}
The multiarrangement $(\Cox(\ell),m_{k, \boldsymbol{n},\Gamma})$ is free if and only if $ \Gamma $ is signed-eliminable. 
\end{theorem}

\begin{proposition}
\label{prop:epsilon'}
Let $V$ be a finite set and let $v\in V$.
Let $ \epsilon =\epsilon_v $ be a map $ \epsilon \colon \binom{V}{2} \to [-1,1] $ such that $ \epsilon(i,v) \geq 0 $ for all $ i \in V\setminus\{v\} $. 
Let $ \epsilon' =\epsilon'_v $ be another map $ \epsilon' \colon \binom{V}{2} \to [-1,1] $ defined by 
$$
\epsilon^{\prime}(i,j)\coloneqq \begin{cases}
 \epsilon(i,v) -1 &  \mbox{if } \,i \in V\setminus\{v\}, \\
\epsilon(i,j) & \mbox{if }\, i,j \in V\setminus\{v\}. 
\end{cases}
$$
Then $ \Gamma(\epsilon) $ is signed-eliminable if and only if $ \Gamma(\epsilon^{\prime}) $ is signed-eliminable. 
\end{proposition}
\begin{proof}
Let $V=[\ell]$.
Note that $(\Cox(\ell),m_{1, (0,\dots,0),\Gamma(\epsilon) }) = (\Cox(\ell),m_{1, \boldsymbol{e}_v,\Gamma(\epsilon' )})$, where $ \boldsymbol{e}_v$ denotes the $\ell$-tuple having $1$ at the $v$th entry and $0$ elsewhere.
Apply Theorem \ref{thm:ANN}.
\end{proof}

We are ready to prove Lemma \ref{coincidence ziegler restrictions}. 

\begin{proof}[\tbf{Proof of Lemma \ref{coincidence ziegler restrictions}}]
In view of Remark \ref{rem:interchange}, it suffices to prove (1).
Since Condition \ref{def:conditions}(\tbf{Z}) is satisfied,
there exist $ n_{0}\in\Z_{\ge0} $, $ \tau(v) \in[0,1] $ and $\tau(i) \in [-1,1] $ for $ i \in [\ell]\setminus\{v\}$ such that  $ |\psi(i)| = 2+n_{0}+\tau(i) $.
Define a map $ \epsilon \colon \binom{[0,\ell]}{2} \to [-1,1] $ by 
\begin{align*}
\epsilon(i,j) \coloneqq \begin{cases}
1 & \mbox{if } \, (i,j),(j,i) \in E_{G}, 1 \leq i < j \leq \ell,  \\
-1 & \mbox{if } \, (i,j),(j,i) \not\in E_{G}, 1 \leq i < j \leq \ell,  \\
\tau(i) & \mbox{if } \,j=0, 1 \leq i \leq \ell, \\
0 & \text{otherwise},
\end{cases}
\end{align*}
Note that $\M$ is identical to the multiarrangement  $(\A,m)$ in $\K^\ell$, where $\A=\Cox(\ell) \cup \Set{ x_{i}=0  | 1 \leq i   \leq \ell}$ and the multiplicity $ m$ is given by 
\begin{align*}
m(x_{i}-x_{j}=0) &= 2+\epsilon(i,j) \qquad (1 \leq i<j \leq \ell), \\
m(x_{i}=0) &=  |\psi(i)| = 2+n_{0}+\epsilon(0,i) \qquad (1 \leq i \leq \ell). 
\end{align*}
Thus $ \mathcal{M} \times \Phi_{1}$ is linearly equivalent\footnote{Two multiarrangements $(\A_1,m_1)$ and $(\A_2,m_2)$  in $\K^\ell$ are said to be \textbf{linearly equivalent} if there is an invertible linear affine endomorphism $\varphi: \K^\ell \to \K^\ell$ such that $\A_2=\varphi(\A_1)$ and $m_1(H)=m_2(\varphi(H))$ for all $H \in \A_1$. 
One can also prove that the freeness of multiarrangements is preserved under linear equivalence.}
 to the multiarrangement $ (\Cox( [0,\ell]),m_{1, (n_{0}, 0,\dots,0),\Gamma(\epsilon)})$ via $ x_{0} \mapsto x_{0} $, $ x_{i} \mapsto x_{i}-x_{0} $ $ (1 \leq i \leq \ell )$. 
Note that $ \Gamma(\epsilon) $ is a signed graph on $ [0,\ell] $. 
Similarly, $\M'$ is identical to the multiarrangement  $(\Cox(\ell),m')$ in $\K^\ell$ where the multiplicity $ m'$ is given by 
\begin{align*}
m'(x_{i}-x_{j}=0) &= m(x_{i}-x_{j}=0) \qquad (i,j \in [\ell]\setminus\{v\}), \\
m'(x_{i}-x_{v}=0) &= m(x_{i}-x_v=0)-1 \qquad (i \in [\ell]\setminus\{v\}), \\
m'(x_{i}=0) &=m(x_{i}=0) +1 \qquad (i \in [\ell]\setminus\{v\}),\\
m'(x_{v}=0) &=  m(x_{v}=0) . 
\end{align*}
Remark that $ \epsilon(i,v) \ge 0$ for all $i \in [0,\ell]\setminus\{v\}$.
Hence
$ \mathcal{M}' \times \Phi_{1}$ is linearly equivalent to $(\Cox( [0,\ell]),m_{1, (n_{0}+1, 0,\dots,0),\Gamma(\epsilon')})$ where $ \epsilon^{\prime}  \colon \binom{[0,\ell]}{2} \to [-1,1] $ is defined by 
$$
\epsilon^{\prime}(i,j)\coloneqq \begin{cases}
 \epsilon(i,v) -1 &  \mbox{if } \,i \in [0,\ell]\setminus\{v\}, \\
\epsilon(i,j) & \mbox{if }\, i,j \in [0,\ell]\setminus\{v\}. 
\end{cases}
$$
Apply Theorem \ref{thm:ANN} and Proposition \ref{prop:epsilon'}, we have
\begin{align*}
\mathcal{M} \text{ is free }
&\Leftrightarrow (\Cox( [0,\ell]),m_{1, (n_{0}, 0,\dots,0),\Gamma(\epsilon)}) \text{ is free } \\
&\Leftrightarrow  \Gamma(\epsilon) \text{ is signed-eliminable} \\
&\Leftrightarrow   \Gamma(\epsilon^{\prime})\text{ is signed-eliminable} \\
&\Leftrightarrow  (\Cox( [0,\ell]),m_{1, (n_{0}+1, 0,\dots,0),\Gamma(\epsilon')}) \text{ is free } \\
&\Leftrightarrow \mathcal{M}^{\prime} \text{ is free}. 
\end{align*}
\end{proof}

\subsection{Local freeness}
\label{subsec:Local Freeness}

Our goal in this subsection is to prove the following lemma.
\begin{lemma}[Stability of local freeness]
\label{coincidence local freeness}
Let $(G,\psi)$ be a vertex-weighted digraph on $ [\ell] $ and let $v$ be a vertex in $G$.
\begin{enumerate}[(1)]
\item Assume that $v $ is a coking in $ G $. 
Let $ (G^{\prime}, \psi^{\prime}) $ be the graph obtained from $(G,\psi)$ by applying the CEO w.r.t. the coking $v$. 
Suppose that Condition \ref{def:conditions}(\tbf{C}) is satisfied. 
Then $\cc\mathcal{A}(G,\psi)$ is  locally free in codimension three along the infinite hyperplane if and only if  $ \cc\mathcal{A}(G^{\prime},\psi^{\prime}) $ is  locally free in codimension three along the infinite hyperplane. 

\item Assume that $ v $ is a king in $ G $. 
Let $ (G'', \psi'') $ be the graph obtained from $(G,\psi)$ by applying the KEO w.r.t. the king $v$. 
Suppose that Condition \ref{def:conditions}(\tbf{K}) is satisfied. 
Then $\cc\mathcal{A}(G,\psi)$ is  locally free in codimension three along the infinite hyperplane if and only if  $ \cc\mathcal{A}(G'',\psi'') $ is  locally free in codimension three along the infinite hyperplane. 
\end{enumerate}
\end{lemma}

The proof of  Lemma \ref{coincidence local freeness} is a brute-force examination on the freeness of all possible codimension-$3$ localizations containing the infinite hyperplane. 
Let us first give the proof plan.
Let $H_{\infty}$ denote the infinite hyperplane $z=0$.
Let $ X $ be an element in $ L(\mathbf{c}\mathcal{A}(G,\psi)) $ such that $H_{\infty}\supseteq X $. 
The following statements hold. 
\begin{itemize}
\item If $ \{x_{i} = az \} \supseteq X $ for some $ a \in \psi(i) $, then $ \{x_{i} = az \} \supseteq X $ for all $ a \in \psi(i) $. 
\item If $ \{x_{i}=az\} \supseteq X $ for some $ a \in \psi(i) $ and $ \{x_{i}-x_{j} = 0\} \supseteq X $, then $ \{x_{j} = az\} \supseteq X $ for any $ j \in [\ell] $. 
\item If $ \{x_{i}-x_{j} = az\} \supseteq X $ for some $ a \in [-1, 1] $, then $ \{x_{i}-x_{j} = az\} \supseteq X $ for all $ a \in [-1, 1]$. 
\item If $ \{x_{i}-x_{j} = 0\} \supseteq X $ and $ \{x_{j}-x_{k} = 0 \} \supseteq X $ for mutually distinct $ i,j,k \in [\ell] $, then $ \{x_{i}-x_{k} = 0\} \supseteq X $. 
\end{itemize}

Thus any possible localization of rank $ 3 $ of $ \mathbf{c}\mathcal{A}(G,\psi) $ containing $H_{\infty}$ is one of the types listed below. Recall that $ G[W] $ denotes the induced subgraph of $ W \subseteq [\ell] $. 
\begin{enumerate}[(L1)]
\item\label{localization type 1} $ \mathbf{c}\mathcal{A}(G[\{i,j\}],\psi|_{\{i,j\}}) $ for distinct $ i,j \in [\ell] $. (Here the assumption $ 0 \in \psi(i) $ for each $ i \in [\ell] $ is necessary.) 
\item\label{localization type 2} $ \mathbf{c}\mathcal{A}(G[\{i,j,k\}],\varnothing) $ for mutually distinct $ i,j,k \in [\ell] $. 
\item\label{localization type 3} $ \mathbf{c}\mathcal{A}(G[\{i,j\}],\varnothing) \cup \mathbf{c}\mathcal{A}(G[\{k\}],\psi|_{\{k\}}) $ for mutually distinct $ i,j,k \in [\ell] $. 
\item\label{localization type 4} $ \mathbf{c}\mathcal{A}(G[\{i,j\}],\varnothing) \cup \mathbf{c}\mathcal{A}(G[\{u,v\}],\varnothing) $ for mutually distinct $ i,j,u,v \in [\ell] $. 
\end{enumerate}

Our task is to characterize the freeness of the arrangements in the four cases above. 
Note that the arrangements of type (L\ref{localization type 3}) and (L\ref{localization type 4}) are supersolvable hence free automatically. 
The arrangement of type (L\ref{localization type 2}) is the one mentioned in Remark \ref{rem:cat} in the case the underlying digraph has $3$ vertices. 
We give the details below. 
\begin{lemma}[{\cite[Lemma 2.1 and Proposition 2.1]{Abe12}}]
 \label{lem:L2}
Let $ G $ be a digraph on $ 3 $ vertices. 
Then $ \mathbf{c}\mathcal{A}(G,\varnothing) $ is free if and only if $ G $ is isomorphic to one of the digraphs in Figure \ref{fig:Type 3}. 
\end{lemma}

 \begin{figure}[htbp!]
\centering
\begin{subfigure}{.2\textwidth}
  \centering
\begin{tikzpicture}[scale=1]
\draw (0,1.3) node[v](1){} node{};
\draw (-0.75,0) node[v](2){} node{};
\draw ( 0.75,0) node[v](3){} node{};
\draw[>=Stealth,->, bend right = 10] (1) to (2);
\draw[>=Stealth,->, bend right = 10] (1) to (3);
\draw[>=Stealth,->, bend right = 10] (2) to (1);
\draw[>=Stealth,->, bend right = 10] (2) to (3);
\draw[>=Stealth,->, bend right = 10] (3) to (1);
\draw[>=Stealth,->, bend right = 10] (3) to (2);
\end{tikzpicture}
 \end{subfigure}%
\begin{subfigure}{.2\textwidth}
  \centering
\begin{tikzpicture}[scale=1]
\draw (0,1.3) node[v](1){} node{};
\draw (-0.75,0) node[v](2){} node{};
\draw ( 0.75,0) node[v](3){} node{};
\draw[>=Stealth,->, bend right = 10] (1) to (2);
\draw[>=Stealth,->, bend right = 10] (1) to (3);
\draw[>=Stealth,->, bend right = 10] (2) to (1);
\draw[>=Stealth,->] (2) to (3);
\draw[>=Stealth,->, bend right = 10] (3) to (1);
\end{tikzpicture}
 \end{subfigure}%
\begin{subfigure}{.2\textwidth}
  \centering
\begin{tikzpicture}[scale=1]
\draw (0,1.3) node[v](1){} node{};
\draw (-0.75,0) node[v](2){} node{};
\draw ( 0.75,0) node[v](3){} node{};
\draw[>=Stealth,->, bend right = 10] (1) to (2);
\draw[>=Stealth,->, bend right = 10] (1) to (3);
\draw[>=Stealth,->, bend right = 10] (2) to (1);
\draw[>=Stealth,->, bend right = 10] (3) to (1);
\end{tikzpicture}
 \end{subfigure}%
 \centering
\begin{subfigure}{.2\textwidth}
  \centering
\begin{tikzpicture}[scale=1]
\draw (0,1.3) node[v](1){} node{};
\draw (-0.75,0) node[v](2){} node{};
\draw ( 0.75,0) node[v](3){} node{};
\draw[>=Stealth,->] (2) to (1);
\draw[>=Stealth,->, bend right = 10] (2) to (3);
\draw[>=Stealth,->] (3) to (1);
\draw[>=Stealth,->, bend right = 10] (3) to (2);
\end{tikzpicture}
 \end{subfigure}%
\begin{subfigure}{.2\textwidth}
  \centering
\begin{tikzpicture}[scale=1]
\draw (0,1.3) node[v](1){} node{};
\draw (-0.75,0) node[v](2){} node{};
\draw ( 0.75,0) node[v](3){} node{};
\draw[>=Stealth,->] (1) to (2);
\draw[>=Stealth,->] (1) to (3);
 \draw[>=Stealth,->, bend right = 10] (2) to (3);
 \draw[>=Stealth,->, bend right = 10] (3) to (2);
\end{tikzpicture}
 \end{subfigure}%

\bigskip
\begin{subfigure}{.2\textwidth}
  \centering
\begin{tikzpicture}[scale=1]
\draw (0,1.3) node[v](1){} node{};
\draw (-0.75,0) node[v](2){} node{};
\draw ( 0.75,0) node[v](3){} node{};
\draw[>=Stealth,->] (1) to (2);
 \draw[>=Stealth,->, bend right = 10] (2) to (3);
 \draw[>=Stealth,->, bend right = 10] (3) to (2);
\end{tikzpicture}
 \end{subfigure}%
\begin{subfigure}{.2\textwidth}
  \centering
\begin{tikzpicture}[scale=1]
\draw (0,1.3) node[v](1){} node{};
\draw (-0.75,0) node[v](2){} node{};
\draw ( 0.75,0) node[v](3){} node{};
\draw[>=Stealth,->, bend right = 10] (2) to (3);
\draw[>=Stealth,->] (3) to (1);
\draw[>=Stealth,->, bend right = 10] (3) to (2);
\end{tikzpicture}
 \end{subfigure}%
\begin{subfigure}{.2\textwidth}
  \centering
\begin{tikzpicture}[scale=1]
\draw (0,1.3) node[v](1){} node{};
\draw (-0.75,0) node[v](2){} node{};
\draw ( 0.75,0) node[v](3){} node{};
\draw[>=Stealth,->] (2) to (1);
\draw[>=Stealth,->] (2) to (3);
\draw[>=Stealth,->] (3) to (1);
\end{tikzpicture}
 \end{subfigure}%
 \begin{subfigure}{.2\textwidth}
  \centering
\begin{tikzpicture}[scale=1]
\draw (0,1.3) node[v](1){} node{};
\draw (-0.75,0) node[v](2){} node{};
\draw ( 0.75,0) node[v](3){} node{};
\draw[>=Stealth,->] (1) to (2);
\draw[>=Stealth,->] (1) to (3);
 \end{tikzpicture}
 \end{subfigure}%
 \begin{subfigure}{.2\textwidth}
  \centering
\begin{tikzpicture}[scale=1]
\draw (0,1.3) node[v](1){} node{};
\draw (-0.75,0) node[v](2){} node{};
\draw ( 0.75,0) node[v](3){} node{};
\draw[>=Stealth,->] (2) to (1);
 \draw[>=Stealth,->] (3) to (1);
\end{tikzpicture}
 \end{subfigure}%
 
 \bigskip
\begin{subfigure}{.2\textwidth}
  \centering
\begin{tikzpicture}[scale=1]
\draw (0,1.3) node[v](1){} node{};
\draw (-0.75,0) node[v](2){} node{};
\draw ( 0.75,0) node[v](3){} node{};
\draw[>=Stealth,->, bend right = 10] (2) to  (3);
\draw[>=Stealth,->, bend right = 10] (3) to  (2);
\end{tikzpicture}
 \end{subfigure}%
\begin{subfigure}{.2\textwidth}
  \centering
\begin{tikzpicture}[scale=1]
\draw (0,1.3) node[v](1){} node{};
\draw (-0.75,0) node[v](2){} node{};
\draw ( 0.75,0) node[v](3){} node{};
\draw[>=Stealth,->] (2) to (3);
\end{tikzpicture}
\end{subfigure}%
\begin{subfigure}{.2\textwidth}
  \centering
\begin{tikzpicture}[scale=1]
\draw (0,1.3) node[v](1){} node{};
\draw (-0.75,0) node[v](2){} node{};
\draw ( 0.75,0) node[v](3){} node{};
\end{tikzpicture}
 \end{subfigure}%
\caption{Characterization for freeness of the arrangements of type (L\ref{localization type 2}).}
\label{fig:Type 3}
\end{figure}
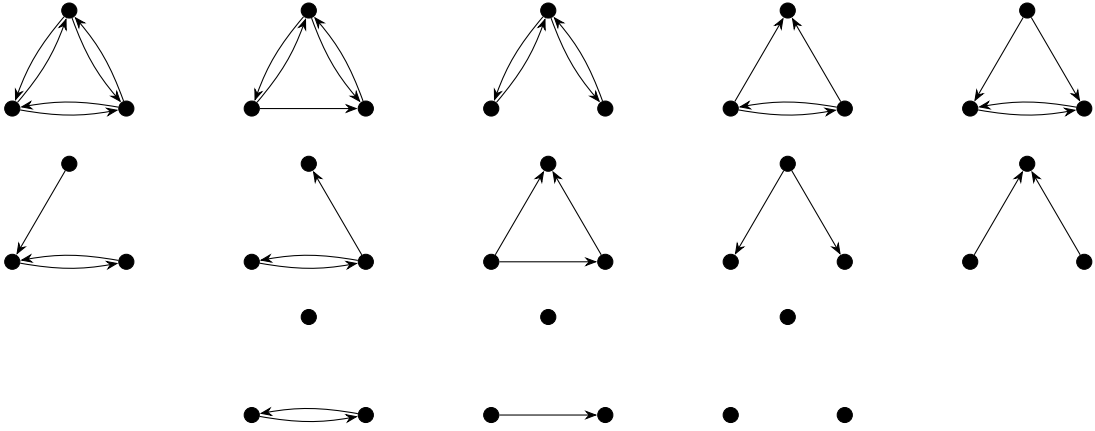

We are left with the arrangement of type (L\ref{localization type 1}). 
In this case $ G[\{i,j\}] $ is isomorphic to either $\overline{K^{\ast}_{2}}$, $ T_{2} $ or $ K^{\ast}_{2} $. 
The following results are useful.
\begin{theorem}[{\cite[Corollary 3.3]{Yo05}}]
 \label{thm:Yo-3arr}
A central arrangement $ \mathcal{A}$ in $\mathbb{K}^3$ is free if and only if 
$$ \chi_{\mathcal{A}}(t) = (t-1)(t-d_2)(t-d_3),$$
where $\exp (\A^H, m^H)=\{d_2, d_3\}$.
\end{theorem}

 \begin{theorem}[{\cite[Theorem 1.5]{W07}}]
 \label{thm:Wakamiko}
Let  $(\A, m)$  be a multiarrangement in $\mathbb{K}^2$ consisting of three lines $K_1,K_2,K_3$. 
Let $k_i:=m(K_i) \in \Z_{\ge0}$ for $1 \le i \le 3$ and suppose $k_3 \ge \max\{k_1,k_2\}$. 
Then
\begin{align*}
\exp (\A, m) = \begin{cases}
\{k_3,k_1+k_2\} & \mbox{ if } k_1+k_2-1\le k_3, \\
\{\lfloor \frac{k_1+k_2+k_3}2 \rfloor, \lceil \frac{k_1+k_2+k_3}2 \rceil\} & \mbox{ if } k_1+k_2-1> k_3. 
\end{cases}
\end{align*}
Here $\lfloor x \rfloor$ and $ \lceil x \rceil$ stand for the floor and ceiling functions, respectively.
\end{theorem}

We now fix some notation that will be used in the remainder of this subsection. 
For a nonempty central arrangement $\A$, it is known that $ \chi_{\mathcal{A}}(t) $ is divisible by $t-1$.
Denote $\overline{\chi}_\A (t)  \coloneqq \frac1{t-1} \chi_{\mathcal{A}}(t) $. 
Let $ \psi \colon [\ell] \to 2^{\mathbb{Z}} $ be an arbitrary weight map. 
For an integer interval $[a,b]$ and for $n \in \Z$, denote $[a,b]+n:=[a+n,b+n]$. 
In particular, $\psi(i)+n=[a_i,b_i]+n=[a_i+n,b_i+n]$.

\begin{proposition}
 \label{prop:K-Nish}
Let $ \mathcal{A}=\mathbf{c}\mathcal{A}(\overline{K^{\ast}_{2}},\psi) $. 
The following are equivalent. 
\begin{enumerate}[(1)]
\item $ \psi(1) \subseteq \psi(2) $ or $ \psi(1) \supseteq \psi(2) $. 
\item $ \mathcal{A} $ is supersolvable. 
\item $ \mathcal{A} $ is free. 
\end{enumerate}
\end{proposition}

\begin{proof}
Note that $\mathcal{A}(\overline{K^{\ast}_{2}},\psi) $ is affinely equivalent to an $N$-Ish arrangement (Example \ref{ex:aff}). The equivalences are already proved in \cite[Theorem 1.3]{AST17}.
\end{proof}

The proofs of the upcoming propositions are very similar. 
We will give the details for two of them hoping that the interested reader will have no difficulty finding the absent proofs. 
Let  $H_c$ denote the hyperplane $x_1-x_2=cz$, where $c \in [-1,1]$.

\begin{proposition}
 \label{prop:T-psi1=2}
Let $ \mathcal{A}=\mathbf{c}\mathcal{A}(T_{2},\psi) $ and suppose  $ |\psi(1)| = |\psi(2)| =m$. 
The following are equivalent. 
\begin{enumerate}[(1)]
\item $ \psi(1) = \psi(2) $ or $ \psi(1) = \psi(2)+1 $. 
\item $ \mathcal{A} $ is free. 
\end{enumerate}
In addition, $ \mathcal{A} $ is supersolvable if and only if $ \mathcal{A} $ is free and $m= 1 $. 
\end{proposition}

\begin{proof}
By definition, $\mathbf{c}\mathcal{A}(T_{2},\psi) $ consists of the following hyperplanes:
\begin{align*}
x_{1} &= az \quad (a \in  \psi(1)), \\
x_{2} &= bz \quad (b \in \psi(2)) ,\\
x_{1}-x_{2}&= 0, z,\\ 
z &= 0. 
\end{align*}
The case $m=0$ is easy. In fact, $\mathcal{A}(T_{2},\varnothing) =\Shi(2)$ (see Example \ref{ex:empty} and Theorem \ref{thm:SI-free}). 

Suppose $m>0$. 
By  Theorem \ref{thm:Wakamiko}, the Ziegler restriction $(\A^{H_{\infty}}, m^{H_{\infty}})$ is free with  $\exp (\A^{H_{\infty}}, m^{H_{\infty}})=\{m+1,m+1\}$. 
It is easily seen that $\mathcal{A} \setminus \{H_0,H_1\}$ is supersolvable with exponents $\{1, m,m\}$.
By the Deletion-Restriction formula (e.g., \cite[Theorem 2.56]{OT92}), 
\begin{align*}
\overline{\chi}_\mathcal{A}(t)  & = \overline{\chi}_{\mathcal{A} \setminus \{H_0,H_1\}}( t)-\overline{\chi}_{(\mathcal{A} \setminus \{H_1\})^{H_0}} ( t)-\overline{\chi}_{\mathcal{A}^{H_1}}( t) \\
  & = (t-m)^2 - (t- |\psi(1) \cup \psi(2)|) - (t- |\psi(1) \cup (\psi(2)+1)|).
\end{align*}

If $ \psi(1) = \psi(2) $ or $ \psi(1) = \psi(2)+1 $, then $\overline{\chi}_\mathcal{A}(t) =  (t-m-1)^2$. 
By Theorem \ref{thm:Yo-3arr}, $ \mathcal{A} $ is free. 
Conversely, if $ \mathcal{A} $ is free, then by Theorem \ref{thm:Yoshinaga's criterion}, $\exp(\A) = \{1,m+1,m+1\}$. 
Therefore, 
\begin{equation*}
 \label{eq:psi1=2}
2m+1 = |\psi(1) \cup \psi(2)| + |\psi(1) \cup (\psi(2)+1)|.
\end{equation*}
Note that $ |\psi(1) \cup (\psi(2)+1)|  \ge  |\psi(1)| = m$. 
Thus $m\le |\psi(1) \cup \psi(2)| \le m+1$.
If $ |\psi(1) \cup \psi(2)|=m$, then $ \psi(1) = \psi(2) $ since  both sets have cardinality $m$. 
If $ |\psi(1) \cup \psi(2)|=m+1$, then $ |\psi(1) \cup (\psi(2)+1)|=m$, hence $ \psi(1) = \psi(2)+1 $. 

Now suppose that $\A$ is  supersolvable. 
Since $\A$ must be free with  $\exp(\A) = \{1,m+1,m+1\}$, by  Theorem \ref{thm:exp-ss}, there exists an M-chain $ \mathcal{A}_{X_1} \subseteq \mathcal{A}_{X_2} \subseteq \mathcal{A}_{X_{3}} = \A,$ where $X_2 \in L(\A)$ is a modular coatom such that $|\A_{X_2}|=m+2$. 
Note that for any coatom $X \in L(\A)$, we have $|\A_X| \in\{3,4,m+1\}$. 
Thus $1\le m\le 2$. 
If $m=2$, then $\A_{X_2} = \{x_1 = az, x_{2}= bz, x_{1}-x_{2}= cz, H_{\infty} \}$ for some $a \in  \psi(1), b \in \psi(2), c\in\{-1,1\}$. 
However, one can use Proposition \ref{prop:ss} to check easily that there is no choice for such $X_1$. 
Thus $m=1$. 
The converse is easy.
\end{proof}

\begin{proposition}
 \label{prop:K-psi1=2}
Let $ \mathcal{A}=\mathbf{c}\mathcal{A}(K^{\ast}_{2},\psi) $ and suppose $ |\psi(1)| = |\psi(2)| =m$. 
The following are equivalent. 
\begin{enumerate}[(1)]
\item either (i) $ \psi(1) = \psi(2) $, or (ii) $m=1, \psi(1) = \psi(2)+1$, or (iii) $m=1, \psi(1) = \psi(2)-1$.
\item $ \mathcal{A} $ is free. 
\end{enumerate}
In addition, $ \mathcal{A} $ is supersolvable if and only if $ \mathcal{A} $ is free and $1\le m\le 2$.  
\end{proposition}
\begin{proof}
By definition, $\mathbf{c}\mathcal{A}(K^{\ast}_{2},\psi) $ consists of the following hyperplanes:
\begin{align*}
x_{1} &= az \quad (a \in  \psi(1)), \\
x_{2} &= bz \quad (b \in \psi(2)) ,\\
x_{1}-x_{2}&= 0, z,-z\\ 
z &= 0. 
\end{align*}
The case $m=0$ is easy. In fact, $\mathcal{A}(K^{\ast}_{2},\varnothing) =\Cat(2)$ (see Example \ref{ex:empty} and Remark \ref{rem:cat}). 

Suppose $m>0$. 
By  Theorem \ref{thm:Wakamiko}, the Ziegler restriction $(\A^{H_{\infty}}, m^{H_{\infty}})$ is free with  $\exp (\A^{H_{\infty}}, m^{H_{\infty}})=\{m+1,m+2\}$. 
It is easily seen that $\mathcal{A} \setminus \{H_0,H_1, H_{-1}\}$ is supersolvable with exponents $\{1, m,m\}$.
By the Deletion-Restriction formula, 
\begin{align*}
\overline{\chi}_\mathcal{A}(t)  & = \overline{\chi}_{\mathcal{A} \setminus \{H_0,H_1,H_{-1}\}}( t)-\overline{\chi}_{(\mathcal{A} \setminus \{H_0,H_1\})^{H_{-1}} }( t)-\overline{\chi}_{(\mathcal{A} \setminus \{H_1\})^{H_0} }( t)-\overline{\chi}_{\mathcal{A}^{H_1} }( t) \\
  & = (t-m)^2 - (t- |\psi(1) \cup (\psi(2)-1)|)\\
   & \quad-(t- |\psi(1) \cup \psi(2)|) - (t- |\psi(1) \cup (\psi(2)+1)|).
\end{align*}

If either (1i), (1ii) or (1iii) occurs, then $\overline{\chi}_\mathcal{A}(t) =  (t-m-1)(t-m-2)$. 
By Theorem \ref{thm:Yo-3arr}, $ \mathcal{A} $ is free. 
Conversely, if $ \mathcal{A} $ is free, then by Theorem \ref{thm:Yoshinaga's criterion},  $\exp(\A)  = \{1,m+1,m+2\}$. 
Therefore, 
\begin{equation}
 \label{eq:kpsi1=2}
3m+2 = |\psi(1) \cup (\psi(2)-1)|+ |\psi(1) \cup \psi(2)| + |\psi(1) \cup (\psi(2)+1)|.
\end{equation}
Note that  $m\le |\psi(1) \cup (\psi(2)-1)|\le m+2$.
\begin{enumerate}[(a)]
\item If $  |\psi(1) \cup (\psi(2)-1)|=m+2$, then $ 2m = |\psi(1) \cup \psi(2)| + |\psi(1) \cup (\psi(2)+1) |$. 
Thus $\psi(2)=\psi(1) = \psi(2)+1$, a contradiction. 
\item If $  |\psi(1) \cup (\psi(2)-1)|=m+1$, then $ 2m+1 = |\psi(1) \cup \psi(2)| + |\psi(1) \cup (\psi(2)+1) |$. 
By the proof of Proposition \ref{prop:T-psi1=2}, $ \psi(1) = \psi(2) $ or $ \psi(1) = \psi(2)+1 $. 
In the case $ \psi(1) = \psi(2)+1 $, we must have  $ |\psi(1) \cup \psi(2)| =m+1= |\psi(1) \cup (\psi(2)-1)|$.
This happens only when $m=1$.
\item If $  |\psi(1) \cup (\psi(2)-1)|=m$, then $\psi(1) = \psi(2)-1$. 
Moreover, $ 2m+2 = |\psi(1) \cup \psi(2)| + |\psi(1) \cup (\psi(2)+1) |$. 
Thus  $ |\psi(1) \cup (\psi(2)+1) |=m+1$.
This happens only when $m=1$.
\end{enumerate}

Now suppose that $\A$ is  supersolvable. 
Since $\A$ must be free with  $\exp(\A) = \{1,m+1,m+2\}$, by  Theorem \ref{thm:exp-ss}, there exists an M-chain $ \mathcal{A}_{X_1} \subseteq \mathcal{A}_{X_2} \subseteq \mathcal{A}_{X_{3}} = \A,$ where $X_2 \in L(\A)$ is a modular coatom such that $m+2 \le |\A_{X_2}|\le m+3$. 
Note that for any coatom $X \in L(\A)$, we have $|\A_X| \in\{4,m+1\}$. 
Thus $1\le m\le 2$ and $ |\A_{X_2}|=4$. The converse is easy.
 \end{proof}

\begin{proposition}
 \label{prop:psi1<2}
Let $ \mathcal{A}=\mathbf{c}\mathcal{A}(T_{2},\psi) $ and suppose $ |\psi(1)| < |\psi(2)| $. 
The following are equivalent. 
\begin{enumerate}[(1)]
\item $ \psi(1) \subseteq \psi(2) \cap (\psi(2)+1) $. 
\item $ \mathcal{A} $ is supersolvable. 
\item $ \mathcal{A} $ is free. 
\end{enumerate}
\end{proposition}

\begin{proposition}
 \label{prop:psi1>2}
Let $ \mathcal{A}=\mathbf{c}\mathcal{A}(T_{2},\psi) $ and suppose $ |\psi(1)| > |\psi(2)| $. 
The following are equivalent. 
\begin{enumerate}[(1)]
\item $ \psi(2) \subseteq \psi(1) \cap (\psi(1)-1) $. 
\item $ \mathcal{A} $ is supersolvable. 
\item $ \mathcal{A} $ is free. 
\end{enumerate}
\end{proposition}

\begin{proposition}
 \label{prop:K-1=2-1}
Let $ \mathcal{A}=\mathbf{c}\mathcal{A}(K^{\ast}_{2},\psi) $ and suppose $m=|\psi(1)| = |\psi(2)|-1 $. 
The following are equivalent. 
\begin{enumerate}[(1)]
\item $ \psi(1) = \psi(2)\cap(\psi(2)+1) $ or $ \psi(1)=\psi(2)\cap(\psi(2)-1) $. 
\item $ \mathcal{A} $ is free.
\end{enumerate}
In addition, $ \mathcal{A} $ is supersolvable if and only if $ \mathcal{A} $ is free and $0\le m\le 1$. 
\end{proposition}

\begin{proposition}
 \label{prop:K*-psi1<2-2}
Let $ \mathcal{A}=\mathbf{c}\mathcal{A}(K^{\ast}_{2},\psi) $ and suppose  $ |\psi(1)| \leq |\psi(2)|-2 $. 
The following are equivalent. 
\begin{enumerate}[(1)]
\item $ \psi(1) \subseteq (\psi(2)-1)\cap \psi(2)\cap(\psi(2)+1) $. 
\item $ \mathcal{A} $ is supersolvable. 
\item $ \mathcal{A} $ is free. 
\end{enumerate}
\end{proposition}

We are ready to prove Lemma \ref{coincidence local freeness}. 

\begin{proof}[\tbf{Proof of Lemma \ref{coincidence local freeness}}]
In view of Remark \ref{rem:interchange}, it suffices to prove (1).
By the discussion up to Lemma \ref{lem:L2}, the assertion holds trivially for the localizations of types (L\ref{localization type 3}) and  (L\ref{localization type 4}). 
By Lemma \ref{lem:L2}, the assertion holds true for the localization of type (L\ref{localization type 2}) as well. 
It is because if $\B$ is a type-(L\ref{localization type 2}) localization of $ \mathbf{c}\mathcal{A}(G,\psi) $ whose underlying digraph admits the coking $v$ as a vertex, then after taking the CEO operation w.r.t. $v$ the resulting type-(L\ref{localization type 2}) localization $\B'$ of $ \mathbf{c}\mathcal{A}(G',\psi') $ also has the underlying digraph isomorphic to one of the digraphs in Figure \ref{fig:Type 3}. 

Now we show that  the assertion holds true for the localization of type (L\ref{localization type 1}). 
That is, we show that $ \mathbf{c}\mathcal{A}(G[\{i,j\}], \psi|_{\{i,j\}}) $ is free if and only if $ \mathbf{c}\mathcal{A}(G^{\prime}[\{i,j\}], \psi^{\prime}|_{\{i,j\}}) $ is free for $ 1 \leq i < j \leq \ell $.  
If  $i\ne v$ and $j\ne v$, then $ \psi^{\prime}(i) = [a_{i}-1, b_{i}] $ and $ \psi^{\prime}(j) = [a_{j}-1, b_{j}] $. 
It is not hard to check by using Propositions \ref{prop:K-Nish}-\ref{prop:K*-psi1<2-2} that the statement holds true.
Now consider $ j = v $ in which case $ G[\{i,v\}] $ is isomorphic to $ T_{2} $ or $ K^{\ast}_{2} $. 

First suppose that $ G[\{i,v\}] $ is isomorphic to $ T_{2} $. 
Using Propositions \ref{prop:T-psi1=2}, \ref{prop:psi1<2}, \ref{prop:psi1>2} and Condition \ref{def:conditions}(\tbf{C}) we have that  $ \mathbf{c}\mathcal{A}(G[\{i,v\}], \psi|_{\{i,v\}}) $ is free if and only if one of the following conditions holds. 
\begin{enumerate}[({A}1)]
\item\label{a1} $ [a_{i}, b_{i}] = [a_{v}, b_{v}] $. 
\item\label{a2} $ [a_{i}, b_{i}] = [a_{v}+1, b_{v}+1] $. 
\item\label{a3} $ [a_{i}, b_{i}] \subseteq [a_{v}+1, b_{v}] $. 
\item\label{a4} $ [a_{i}, b_{i}-1] = [a_{v}, b_{v}] $. 
\end{enumerate}
In this case $G^{\prime}[\{i,v\}]$ is  isomorphic to $ \overline{K^{\ast}_{2}} $. 
Thus by Proposition \ref{prop:K-Nish}, $ \mathbf{c}\mathcal{A}(G^{\prime}[\{i,v\}], \psi^{\prime}|_{\{i,v\}}) $ is free if and only if one of the following conditions holds. 
\begin{enumerate}[({B}1)]
\item\label{b1} $ [a_{i}-1, b_{i}] \subseteq [a_{v}, b_{v}] $. 
\item\label{b2} $ [a_{i}-1, b_{i}] \supseteq [a_{v}, b_{v}] $. 
\end{enumerate}
The following implications are straightforward: 
(A\ref{a1}) $ \Rightarrow $ (B\ref{b2}), 
(A\ref{a2}) $ \Rightarrow $ (B\ref{b2}), 
(A\ref{a3}) $ \Leftrightarrow $ (B\ref{b1}), 
and (A\ref{a4}) $ \Rightarrow $ (B\ref{b2}). 
Furthermore, Condition \ref{def:conditions}(\tbf{C}) and (B\ref{b2}) imply $ [a_{i}, b_{i}-1] \subseteq [a_{v}, b_{v}] \subseteq [a_{i}-1, b_{i}] $. 
Hence $ [a_{v}, b_{v}] $ equals either $ [a_{i}, b_{i}], [a_{i}-1, b_{i}-1], [a_{i}-1, b_{i}] $, or $ [a_{i}, b_{i}-1] $. 
These lead to (A\ref{a1}), (A\ref{a2}), (A\ref{a3}), and (A\ref{a4}), respectively. 
Thus the assertion holds true. 

Next suppose that $ G[\{i,v\}] $ is isomorphic to $ K^{\ast}_{2} $. 
Thus, $ \mathbf{c}\mathcal{A}(G[\{i,v\}], \psi|_{\{i,v\}}) $ is free if and only if one of the following conditions holds. 
\begin{enumerate}[({C}1)]
\item\label{c1} $ [a_{i}, b_{i}] = [a_{v}, b_{v}] $. 
\item\label{c2} $ [a_{i}, b_{i}] = [a_{v}+1, b_{v}] $. 
\item\label{c3} $ [a_{i}, b_{i}] = [a_{v}, b_{v}-1] $. 
\item\label{c5} $ [a_{v}, b_{v}] = [a_{i}, b_{i}-1] $. 
\item\label{c6} $ [a_{i}, b_{i}] \subseteq [a_{v}+1, b_{v}-1] $.
\end{enumerate}
Again Condition \ref{def:conditions}(\tbf{C}) is crucial here.
(C\ref{c1}) is derived from Proposition \ref{prop:K-psi1=2} noting that $m=1$ can not occur since  $ 0 \in \psi(i) $ for all $ i \in [\ell] $. 
Note also that $ [a_{i}, b_{i}-1] \subseteq [a_{v}, b_{v}]$, in particular, $ |\psi(i)| -1 \leq |\psi(v)| $ for all $ i \ne v$. 
(C\ref{c5}) and  (C\ref{c2})\&(C\ref{c3}) are derived from Proposition \ref{prop:K-1=2-1} according to $ |\psi(i)| -1 = |\psi(v)| $ or $ |\psi(v)| -1 = |\psi(i)| $, respectively. 
Finally, (C\ref{c6})  is derived from Proposition \ref{prop:K*-psi1<2-2}.
In this case $G^{\prime}[\{i,v\}]$ is  isomorphic to $T_2 $. 
Note that $ \psi^{\prime}(i) = [a_{i}-1, b_{i}] $ and $ \psi^{\prime}(v) = [a_v, b_{v}] $. 
Thus, $ \mathbf{c}\mathcal{A}(G^{\prime}[\{i,v\}], \psi^{\prime}|_{\{i,v\}}) $ is free if and only if one of the following conditions holds. 
\begin{enumerate}[({D}1)]
\item\label{d1} $ [a_{v}, b_{v}] = [a_{i}-1, b_{i}] $. 
\item\label{d2} $ [a_{v}, b_{v}] = [a_{i}, b_{i}+1] $. 
\item\label{d3} $ [a_{v}, b_{v}] \subseteq [a_{i}, b_{i}] $.
\item\label{d4} $ [a_{v}, b_{v}-1] \supseteq [a_{i}-1, b_{i}] $.  
\end{enumerate}
The following implications are straightforward: 
(C\ref{c1}) $ \Rightarrow $ (D\ref{d3}), 
(C\ref{c2}) $ \Leftrightarrow $ (D\ref{d1}), 
(C\ref{c3}) $ \Leftrightarrow $ (D\ref{d2}), 
(C\ref{c5}) $ \Rightarrow $ (D\ref{d3}), 
and (C\ref{c6}) $ \Leftrightarrow $ (D\ref{d4}).
Furthermore, Condition \ref{def:conditions}(\tbf{C}) and  (D\ref{d3}) imply $ [a_{i}, b_{i}-1] \subseteq [a_{v}, b_{v}] \subseteq [a_{i}, b_{i}] $. 
Hence $ [a_{v}, b_{v}] $ is equal to $ [a_{i}, b_{i}] $ or $ [a_{i}, b_{i}-1] $. 
These are (C\ref{c1}) and (C\ref{c5}), respectively. 
Therefore the assertion holds true. 
\end{proof}

\subsection{Proof of Theorem \ref{thm:freeness-stable} and applications}
\label{subsec:freeness-stable}
We are ready to prove Theorems \ref{thm:freeness-stable} and \ref{main1}. 

\begin{proof}[\tbf{Proof of  Theorem \ref{thm:freeness-stable}}]
It follows from Theorem \ref{thm:Yoshinaga's criterion}, Lemma \ref{coincidence ziegler restrictions} and Lemma \ref{coincidence local freeness}.
\end{proof}

\begin{proof}[\tbf{Proof of  Theorem \ref{main1}}]
By Proposition \ref{prop:Akl modified}, $ \mathcal{A}_{\ell+1}^{k+1} \aff \mathcal{A}(T_{\ell}^{k}, \psi_{\ell}^{k}) \times \Phi_{1} $. 
Hence it is sufficient to prove that $ \mathbf{c}\mathcal{A}(T_{\ell}^{k}, \psi_{\ell}^{k}) $ if free for each $ k \in [\ell] $. 
We will proceed by a downward induction on $ k $. 
When $ k = \ell $, by \cite[Theorem 1.3]{AST17} or Proposition \ref{prop:simplicial}, the cone over $ \Ish(\ell+1) = \mathcal{A}_{\ell+1}^{\ell} \aff \mathcal{A}(T_{\ell}^{\ell},\psi_{\ell}^{\ell}) \times \Phi_{1} $ and $ \mathbf{c}\mathcal{A}(T_{\ell}^{\ell},\psi_{\ell}^{\ell}) $ are supersolvable and hence $ \mathbf{c}\mathcal{A}(T_{\ell}^{\ell},\psi_{\ell}^{\ell}) $ is free. 

Suppose that $ k < \ell $ and $ \mathcal{A}(T_{\ell}^{k+1},\psi_{\ell}^{k+1}) $ is free. 
By Proposition \ref{prop:simplicial} and Proposition \ref{prop:isolated}, $ \mathcal{A}(T_{\ell}^{k+1}[\ell-k+1],\psi_{\ell}^{k+1}|_{[\ell-k+1]}) $ is free. 
Recall the CEO $ ((T_{\ell}^{k}[\ell-k+1])^{\prime}, (\psi_{\ell}^{k}|_{[\ell-k+1]})^{\prime}) = (T_{\ell}^{k+1}[\ell-k+1], \psi_{\ell}^{k+1}|_{[\ell-k+1]}) $ in Theorem \ref{thm:Akl} with respect to the coking $ v = \ell-k+1 $ in $ T_{\ell}^{k}[\ell-k+1] $. 
Since for any $ i \in [\ell-k+1] $, $ |\psi_{\ell}^{k}(i)| = k+1 $, Condition \ref{def:conditions}(\tbf{Z}) is satisfied with $ n_{0} = k-1 $. 
Since Condition \ref{def:conditions}(\tbf{C}) is also satisfied, by Theorem \ref{thm:freeness-stable}, $ \mathbf{c}\mathcal{A}(T_{\ell}^{k}[\ell-k+1], \psi_{\ell}^{k}|_{[\ell-k+1]}) $ is free.
By Proposition \ref{prop:simplicial} and Proposition \ref{prop:isolated} again, $ \mathcal{A}(T_{\ell}^{k},\psi_{\ell}^{k}) $ is free. 
Thus, $ \mathbf{c}\mathcal{A}(T_{\ell}^{k},\psi_{\ell}^{k}) $ is free for every $ k \in [\ell] $. 
Therefore $ \mathbf{c}\mathcal{A}_{\ell}^{k} $ is free for each $ k \in [2,\ell] $. 

%
%
%
%
%

Now we show that if $\ell \ge 3$ then $\cc \mathcal{A}_{\ell}^{k}$  is not supersolvable for $ k \in [2,\ell -1]$. 
It suffices to prove that $ \mathbf{c}\mathcal{A}(T_{\ell}^{k}, \psi_{\ell}^{k}) $ is not supersolvable when $ \ell \geq 2 $ and $ k \in [\ell-1] $. 
The key point here is that $(1,2)$ is always an edge in $T_{\ell}^{k}$ for $ k < \ell $. 
Let $\B$ be the localization of $\textbf{c} \mathcal{A}(T_{\ell}^{k},\psi_{\ell}^{k}) $ on the subspace defined by $ z=x_{1}=x_2=0$.  
Thus $\B=\mathbf{c}\mathcal{A}(T_{2},[-k,0])$ which is not supersolvable by  Proposition \ref{prop:T-psi1=2}. 
Theorem \ref{thm:localization-ss-free} completes the proof.
\end{proof}
 
In particular,  if $\ell \ge 3$ then the intersection posets $ L(\cc \mathcal{A}_{\ell}^{k}) $ and $ L(\cc\Ish(\ell)) $ are not isomorphic for any $2 \leq k \leq \ell-1$. In the theorem below, we show that a similar phenomenon occurs in the Catalan arrangement.

For any $ k \in [\ell] $, define vertex-weighted digraphs $ (C_{\ell}^{k}, \psi_{\ell}^{k}) $ and $ (D_{\ell}^{k}, \phi_{\ell}^{k}) $ on $ [\ell] $ by
\begin{align*}
E_{C_{\ell}^{k}} &\coloneqq \Set{(i,j) | i,j \in [k,\ell], \, i\neq j}, \\
\psi_{\ell}^{k}(i) &\coloneqq [-\min\{i,k\}, \min\{i,k\}], \\
E_{D_{\ell}^{k}} &\coloneqq \Set{(k,i) | i \in [k+1, \ell]} \cup \Set{(i,j) | i,j \in [k+1, \ell], \, i\neq j}, \\
\phi_{\ell}^{k}(i) &\coloneqq \begin{cases}
[-\min\{i,k\}, \min\{i,k\}] & ( i \leq k), \\
[-\min\{i,k\}-1, \min\{i,k\}] & (i > k). 
\end{cases}
\end{align*}
Note that $ \Cat(\ell+1) = \mathcal{A}(K^{\ast}_{\ell}, [-1,1]) = \mathcal{A}(C_{\ell}^{1}, \psi_{\ell}^{1}) $ and $ \mathcal{A}(C_{\ell}^{\ell}, \psi_{\ell}^{\ell}) = \mathcal{A}(\overline{K^{\ast}_{\ell}}, [-i,i]) $ has a supersolvable cone with exponents $ \{1,\ell+2,\ell+3, \dots, 2\ell+1\} $

\begin{theorem}[Catalan arrangement] 
 \label{thm:Cat-CKEO}
Applying the CEO with respect to the coking $ k $ to $ (C_{\ell}^{k},\psi_{\ell}^{k}) $ yields $ (D_{\ell}^{k}, \phi_{\ell}^{k}) $ and applying the KEO with respect to the king $ k $ to $ (D_{\ell}^{k}, \phi_{\ell}^{k}) $ yields $ (C_{\ell}^{k+1},\psi_{\ell}^{k+1}) $. 
Hence there exists a sequence of CEOs and KEOs applying to  $ \Cat(\ell)$ that  preserves both characteristic polynomials and freeness and ends with a $ \psi $-digraphical arrangement having supersolvable cone whose underlying digraph is edgeless. 
In particular,  we obtain a new proof of the well-known fact that $\cc\Cat(\ell)$ is  free with exponents $\{0,1, \ell+1, \ell+2,\ldots,2\ell-1\}$. 
 \end{theorem}
\begin{proof}
Similar to the proofs of Theorem  \ref{thm:Akl} and Theorem \ref{main1}. 

\end{proof}

The sequence mentioned in Theorem \ref{thm:Cat-CKEO} in the case $\ell=4$ is depicted in Figure \ref{fig:Cat}. 

 \begin{figure}[htbp!]
\centering
\begin{subfigure}{.35\textwidth}
  \centering
\begin{tikzpicture}[scale=1]
\draw (0,1.3) node[v](1){} node[above]{$ \textbf{1} [-1,1]$};
\draw (-0.75,0) node[v](2){} node[left]{$ \textbf{2}[-1,1]$};
\draw (0.75,0) node[v](3){} node[right]{$ \textbf{3}[-1,1]$};
\draw[>=Stealth,->, bend right = 10] (1) to (2);
\draw[>=Stealth,->, bend right = 10] (1) to (3);
\draw[>=Stealth,->, bend right = 10, red] (2) to (1);
\draw[>=Stealth,->, bend right = 10] (2) to (3);
\draw[>=Stealth,->, bend right = 10, red] (3) to (1);
\draw[>=Stealth,->, bend right = 10] (3) to (2);
\end{tikzpicture}
  \caption*{$ \cc\Cat(4)  \aff  \cc\mathcal{A}(K^{\ast}_{3},[-1,1]) $ is not supersolvable}
\end{subfigure}%
\begin{subfigure}{.35\textwidth}
  \centering
\begin{tikzpicture}[scale=1]
\draw (0,1.3) node[v](1){} node[above]{$ \textbf{1}[-1,1]$};
\draw (-0.75,0) node[v](2){} node[left]{$ \textbf{2}[-2,1]$};
\draw (0.75,0) node[v](3){} node[right]{$ \textbf{3}[-2,1]$};
\draw[>=Stealth,->,red] (1) to (2);
\draw[>=Stealth,->,red] (1) to (3);
\draw[>=Stealth,->, bend right = 10] (2) to (3);
\draw[>=Stealth,->, bend right = 10] (3) to (2);
\end{tikzpicture}
  \caption*{not supersolvable}
\end{subfigure}%
\begin{subfigure}{.35\textwidth}
  \centering
\begin{tikzpicture}[scale=1]
\draw (0,1.3) node[v](1){} node[above]{$ \textbf{1}[-1,1]$};
\draw (-0.75,0) node[v](2){} node[left]{$ \textbf{2}[-2,2]$};
\draw (0.75,0) node[v](3){} node[right]{$ \textbf{3}[-2,2]$};
\draw[>=Stealth,->, bend right = 10] (2) to (3);
\draw[>=Stealth,->,red, bend right = 10] (3) to (2);
\end{tikzpicture}
  \caption*{not supersolvable}
\end{subfigure}%

\bigskip
\begin{subfigure}{.45\textwidth}
  \centering
\begin{tikzpicture}[scale=1]
\draw (0,1.3) node[v](1){} node[above]{$ \textbf{1}[-1,1]$};
\draw (-0.75,0) node[v](2){} node[left]{$ \textbf{2}[-2,2]$};
\draw (0.75,0) node[v](3){} node[right]{$ \textbf{3}[-3,2]$};
\draw[>=Stealth,->,red] (2) to (3);
\end{tikzpicture}
  \caption*{supersolvable}
\end{subfigure}%
\begin{subfigure}{.45\textwidth}
  \centering
\begin{tikzpicture}[scale=1]
\draw (0,1.3) node[v](1){} node[above]{$ \textbf{1}[-1,1]$};
\draw (-0.75,0) node[v](2){} node[left]{$ \textbf{2}[-2,2]$};
\draw (0.75,0) node[v](3){} node[right]{$ \textbf{3}[-3,3]$};
\end{tikzpicture}
  \caption*{$ \cc\A  (\overline{K^{\ast}_{3}}, [-i,i])$ is supersolvable with exponents $\{1, 5, 6,7\}$}
\end{subfigure}
\caption{A sequence of CEOs and KEOs applying to the Catalan arrangement $\Cat(4)$.}
\label{fig:Cat}
\end{figure}
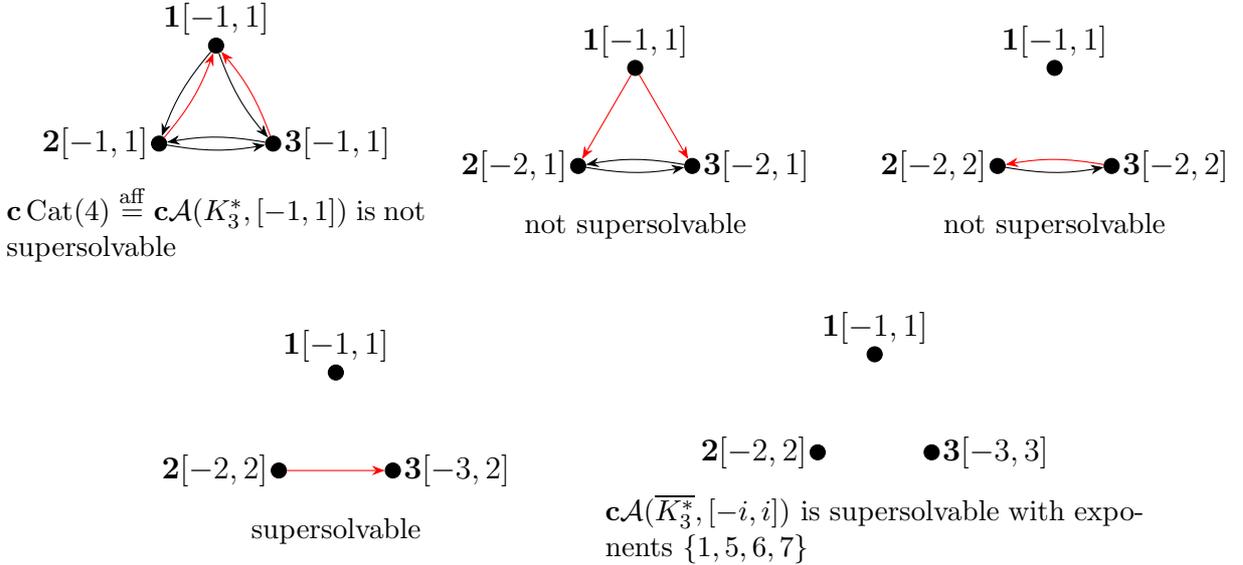

\vskip 1em
\noindent
\textbf{Acknowledgements.} 
We would like to thank Professor Masahiko Yoshinaga for suggesting to us the question of freeness for the arrangements between Shi and Ish. 
The first author is partially supported by JSPS KAKENSHI grant number JP20K20880 and 21H00975.
 The second author was supported by JSPS Research Fellowship for Young Scientists Grant Number 19J12024, and is currently supported by a Postdoctoral Fellowship of the Alexander von Humboldt Foundation.

\bibliographystyle{amsplain}
\bibliography{references}

\end{document}